\def\marker{\>\hbox{${\vcenter{\vbox{
					\hrule height 0.4pt\hbox{\vrule width 0.4pt height 6pt
						\kern6pt\vrule width 0.4pt}\hrule height 0.4pt}}}$}\>}

\documentclass[12pt]{article}
\usepackage{indentfirst}
\usepackage{epsfig}
\usepackage{fullpage}
\usepackage{amsmath,amsfonts,amssymb,amsthm,graphics,amscd}
\usepackage{pgf,tikz}
\usepackage{graphicx}
\usepackage{subfigure}
\usepackage{amsmath,amsfonts,amssymb,amsthm,graphics,amscd,pst-plot,pstricks}
\usepackage{float}
\usepackage{graphicx}

\usepackage{xcolor}
\usepackage{caption}
\newtheorem {Theorem}  {Theorem}
\newtheorem {Lemma}{Lemma}[section]

\newtheorem {Proposition}[Lemma]{Proposition}

\theoremstyle{definition}
\newtheorem{Definition}{Definition}

\newtheorem{Conjecture}{Conjecture}

\newcommand{\D}{\Delta}

\newcommand{\phiv}{\varphi}

\newcommand{\CC}{\mathcal{C}}

\usepackage[
pdfauthor={},
pdftitle={Edge-chromatic $4$-critical graphs and Overfull Conjecture for graphs with maximum degree $4$},
pdfstartview=XYZ,
bookmarks=true,
colorlinks=true,
linkcolor=blue,
urlcolor=blue,
citecolor=blue,
bookmarks=false,
linktocpage=true,
hyperindex=true
]{hyperref}

\renewcommand{\thefootnote}
\begin{document}
\begin{center}
		
{\Large \bf Edge-chromatic $4$-critical graphs and Overfull Conjecture for graphs with maximum degree $4$}
	
\vspace{10mm}

{\large Chunhui Ge$^{a}$, Gregory Gutin$^{b}$, Xuli Qi$^{a}$}

\vspace{4mm}

\baselineskip=0.25in

{\it  $^a$ Department of Mathematics, Hebei University of Science and Technology, China.}\\
{\it  $^b$ Department of Computer Science, Royal Holloway University of London, UK.}\\

\vspace{6mm}		
\end{center}

\footnotetext{{\it E-mail addresses:} {\tt 2024110004@stu.hebust.edu.cn}; {\tt g.gutin@rhul.ac.uk}; {\tt xuliqi@hebust.edu.cn}}

\date{}
	
\begin{abstract}

Let $G$ be a simple graph with maximum degree $\D(G)$ and chromatic index $\chi'(G)$. A graph $G$ is called edge-chromatic {\em $\D$-critical} if $\chi'(G)=\D(G)+1$ and $\chi'(H)< \chi'(G)$ for every proper subgraph $H$ of $G$, and $G$ is overfull if $\left|E(G)\right|>\Delta(G)\lfloor |V(G)|/2\rfloor$.
In 1986, Chetwynd and Hilton proposed the influential Overfull Conjecture: If $G$ is a simple graph with $\D(G)>\frac{|V(G)|}{3}$, then $G$ is a Class $2$ graph if and only if $G$ contains an overfull subgraph $H$ with $\D(H)=\D(G)$.
Motivated by the structural analysis for $4$-critical graphs in \cite{CR2019} (SIAM J. Discrete Math. 2019), we show more properties in this paper,
especially four new forbidden configurations in any $4$-critical graph, and provide a new structural proof of  Overfull Conjecture for graphs with maximum degree $4$.	
	\end{abstract}

\section{Introduction}
	In this paper we only consider simple and connected graphs.
	We generally follow the book \cite{SSTF2012} of Stiebitz et al. for notation and terminology.
	Let $G=(V(G),E(G))$ be a simple graph with vertex set $V(G)$ and edge set $E(G)$.
	We denote by $\Delta(G)$ and $\delta(G)$ (or simply $\Delta$ and $\delta$), respectively the maximum degree and minimum degree of graph $G$. A graph $G$ is $\Delta$-regular if $\delta(G) = \Delta(G) = \Delta$. Let $[k]:=\{1,\dots,k\}$.	
	A (proper) \emph{$k$-edge-coloring}  of a graph $G$ is a map $\phiv:E(G)\to [k]$ that assigns to every edge $e$ of $G$ a color $\phiv(e)\in [k]$ such that no two adjacent edges of $G$ receive the same color.
	Let $\CC^k(G)$ denote the set of all $k$-edge-colorings of $G$.
	The \emph{chromatic index} $\chi'(G)$ is the least integer $k$ such that $\CC^k(G)\ne\emptyset$.
	A graph	$G$ is \emph{edge-chromatic critical} if $\chi'(H)<\chi'(G)$ for every proper subgraph $H$ of $G$. Furthermore, we call $G$ a \emph{$\D$-critical} graph if $G$ is edge-chromatic critical and $\chi'(G)=\D(G)+1$.
	
	Vizing \cite{Vizing1965-1} and Gupta \cite{Gupta1967} proved that $\Delta(G)\le \chi'(G)\le\Delta(G)+1$.
	Following Fiorini and Wilson \cite{FW1977}, if $\chi'(G)=\Delta(G)$, then  we call $G$ a \emph{Class 1} graph; otherwise, we call it a \emph{Class 2} graph.
	If a graph $G$ satisfies $\left|E(G)\right|> \Delta(G)\lfloor |V(G)|/2\rfloor$, then it is called \emph{overfull} and every edge-coloring of $G$ uses at least $\D(G)+1$ colors.
	Holyer \cite{Holyer1981} proved that it is NP-complete to determine whether an arbitrary graph is Class $1$.
	In contrast, Chetwynd and Hilton \cite{CHETWYND1986} conjectured that there is a polynomial-time algorithm to determine the chromatic index for graphs $G$ with $\D(G)>\frac{|V(G)|}{3}$. And the Overfull Conjecture of Chetwynd and Hilton \cite{CHETWYND1986,CHETWYND1989} is as follows.
	
	\begin{Conjecture}[Overfull Conjecture]
		Let $G$ be a graph with $\D(G)>\frac{|V(G)|}{3}$. Then $G$ is a Class $2$ graph if and only if $G$ contains an overfull subgraph $H$ with $\D(H)=\D(G)$.
	\end{Conjecture}

    It suffices to consider connected graphs: if $G$ is disconnected and Class 2, then some component contains a $4$-critical subgraph $H$, and the subsequent argument applies to $H$.
	
	Note that the degree condition  $\D(G)>\frac{|V(G)|}{3}$ in the conjecture above is best possible since for the graph $P^*$ obtained from the Petersen graph by deleting one vertex, $\chi'(P^*)=4$ and $\D(P^*)=\frac{|V(P^*)|}{3}=3$ but with no overfull subgraphs.
	Seymour \cite{Seymour1979} showed that deciding whether a graph $G$ contains an overfull subgraph with maximum degree $\D(G)$ can be determined in polynomial time.
	Furthermore, the Overfull Conjecture can also imply several other longstanding conjectures such as the Just Overfull Conjecture \cite{SSTF2012}, the Vertex-splitting Conjecture \cite{HZ1997}, and Vizing's 2-factor Conjecture, Independence Conjecture and Average Degree Conjecture \cite{Vizing1965-2,Vizing1968} when restricted to graphs $G$ with $\D(G)>\frac{|V(G)|}{3}$.
	
	There are not many direct research results with the maximum degree constraints for the Overfull Conjecture.
	For large maximum degree, Chetwynd and Hilton \cite{CHETWYND1989} proved that the conjecture is true for all graphs $G$ with $\D(G)\ge |V(G)|-3$.
	Recently, Shan \cite{Shan2024-2} confirmed that for any $0<\varepsilon\le\frac{1}{14}$, there exists a positive integer $n_0$ such that if $G$ is a graph with  $|V(G)|\ge n_0$ and $\D(G)\ge(1-\varepsilon)|V(G)|$, then the conjecture holds.
	In the small order and small maximum degree direction,  Chetwynd and Yap \cite{CC1983}  verified in 1983 that the Overfull Conjecture holds for all graphs of order at most $10$. In 1998, Brinkmann and Steffen \cite{BS1998}  used computational methods and showed that there are exactly two non-overfull critical graphs of order $11$, both of which are $3$-critical. Combining this with \cite{CC1983}, one can deduce that the Overfull Conjecture holds when $\Delta=4$.  In 2008, Pi \cite{Pi}   gave a theoretical proof of the Overfull Conjecture for the case $\Delta=4$. However, there are several gaps in the proofs of main theorem and lemmas in \cite{Pi}. Specifically, several errors appear in their analyses concerning the endpoints of Kempe chains. Motivated by \cite{CR2019}, we obtain several structural properties of $4$-critical graphs, especially four new forbidden configurations in any $4$-critical graph, i.e.,
     Propositions \ref{P2.13}-\ref{P2.16} listed in the next section of this paper, which are instrumental for a new structural proof of the Overfull Conjecture for the case $\Delta=4$.
	\begin{Theorem}\label{T1}
		Let $G$ be a graph of order $n$ with $\Delta(G)=4 >\frac{n}{3}.$ Then
		$G$ is a Class $2$ graph if and only if $G$ contains an overfull subgraph $H$ with $\D(H)=4$.
	\end{Theorem}

\section{Definitions and preliminary results}
	
	Let $G$ be a graph. For each edge $e \in E(G)$, let $G-e$ denote the graph obtained by
	deleting the edge $e$ from $G$. An edge $e$ is called a \emph{critical edge} of $G$ if
	$\chi'(G-e)<\chi'(G)$. A graph $G$
	is {\em edge-chromatic critical} if $G$ is a Class $2$ graph and every edge is critical. For $u \in V(G)$, denote by $d_G(u)$ the degree of $u$ in $G$.
	For adjacent vertices $u,v\in V(G)$, we call $(u,v)$ a \emph{full-deficiency pair} of $G$ if
	$d_G(u)+d_G(v)=\Delta(G)+2$. A $k$-{\em vertex} in $G$ is a vertex with degree $k$.
	For $u\in V(G)$, a $k$-{\em neighbor} of $u$ is a neighbor of $u$ that is a $k$-vertex in $G$. Let $n_i$ denote the number of $i$-vertices in $G$ with order $n$ for $i \in [\Delta]$. Let $N_G(u)$ denote the set of neighbors of $u \in V(G)$. For $u,v \in V(G)$, let $dist_G(u,v)$ denote the distance between $u$ and $v$, defined as the length of a shortest path connecting $u$ and $v$ in $G$. For $S \subseteq V(G)$, we define $dist_G(u,S)=\min_{v \in S}dist_G(u,v)$. For $T \subseteq V(G)$, let $G[T]$ denote the induced subgraph on $T$, that is, the subgraph with vertex set $T$ and each edge of $G$ with both endpoints in $T$.
	
	Let $G$ be a graph with an edge $e\in E(G)$, and $\phiv\in \CC^{k}(G-e)$ be an edge coloring
	for some integer $k\ge\D$. For $v \in V(G)$, let $\phiv(v)$ be the set of colors
	assigned to edges incident with $v$, and let $\overline{\varphi}(v) = [k]\setminus\phiv(v)$
	be the set of colors not assigned to any edge incident with $v$. The set $\phiv(v)$ is called the
	set of colors \emph{present} at $v$ and $\overline{\varphi}(v)$ is called the set of colors \emph{missing} at $v$. It is easy to see that
	$|\phiv(v)| + |\overline{\varphi}(v)| =k$ for each vertex $v\in V(G)$.
	For $X\subseteq V(G)$, we define $\overline{\varphi}(X)=\bigcup_{v\in X}
	\overline{\varphi}(v)$.
	The set $X$ is {\em elementary with respect to} $\phiv$ (or simply \emph{$\phiv$-elementary}),
	if $\overline{\varphi}(u)\cap\overline{\varphi}(v)=\emptyset$ for any two distinct vertices
	$u,v\in X$.
	Let $E_{\varphi,\alpha}(G)$ denote the set of edges colored
	with a color $\alpha\in [k]$.

	For any two distinct colors $\alpha,\beta\in[k]$, let $H$ be the  subgraph induced
	by $E_{\varphi,\alpha}(G)$ and $E_{\varphi,\beta}(G)$.
	Every component of $H$ is either a path or an even cycle, each such component is called an
	\emph{$(\alpha,\beta)$-chain} of $G$ (or simply a {\em Kempe} chain). If an $(\alpha,\beta)$-chain of $G$ contains both vertices $x$ and $y$ with respect to $\varphi$, then $x$ and $y$ are \emph{ $(\alpha,\beta)$-linked}.
	Swapping the colors $\alpha$ and $\beta$ along an $(\alpha,\beta)$-chain $C$ gives a
	new (proper) $k$-edge-coloring of $G$,
	which also belongs to
	$\mathcal{C}^k(G-e)$. This operation is called a \emph{Kempe change}. If $C$ is a path with an endvertex $u$, we call this operation an 	\emph{$(\alpha,\beta)$-swap} at $u$.
	Let $\alpha \in \overline{\varphi}(u)$ and $\beta,\gamma \in \varphi(u)$.
	An $(\alpha,\beta)-(\beta, \gamma)$-swaps at $u$ consists of two consecutive changes: we first swap the colors along $P_u(\alpha,\beta,\varphi)$ to get a $k$-edge-coloring $\varphi'$, and then swap the colors along  $P_u(\beta,\gamma,\varphi')$.
	For $v \in V(G)$, let $P_v(\alpha,\beta,\varphi)$ or simply $P_v(\alpha,\beta)$ denote the unique
	$(\alpha,\beta)$-chain containing  $v$.
	{\em For any two distinct vertices $u,v\in V(G)$, $P_u(\alpha,\beta,\varphi)$ and $P_v(\alpha,\beta,\varphi)$ are either identical or
		disjoint, which is a fundamental property in our proofs.}
	Let $P_{[a,b]}(\alpha, \beta,\varphi)$ denote the subchain of an $(\alpha,\beta)$-chain $P$ with endvertices $a$ and $b$.
	Swapping the colors $\alpha$ and $\beta$ along the subchain $P_{[a,b]}(\alpha,
	\beta,\varphi)$ is still called a Kempe change, but the resulting edge-coloring may no longer be proper.
	We write $uv:\alpha\rightarrow\beta$ to denote recoloring the edge $uv$ from $\alpha$ to $\beta$.
	For $x,y \in P_u(\alpha,\beta,\phiv)$, if $x$  lies between $u$ and $y$ along the path, then we say that $P_u(\alpha,\beta,\phiv)$ meets $x$ before $y$.

	\begin{Definition}[Multi-fan]
		Let $G$ be a graph with an edge $e_1=xy_1$, and an edge coloring $\phiv\in \CC^{k}(G-e_1)$
		for some integer $k\ge\D(G)$.  A multi-fan at $x$ with respect to $e_1$ and $\phiv$ is a
		sequence $F=(x,e_1,y_1,\dots,e_p,y_p)$ with $p \ge 1$ consisting of edges $e_1, \dots,e_p$
		and vertices $x,y_1, \dots ,y_p$ satisfying the following two conditions$\colon$
		\begin{itemize}
			\item [{ F1.}]	The edges $e_1, \dots, e_p$ are distinct,  and $e_i =xy_i$
			for $1\le i\le p$.
			\item [{ F2.}] For every edge $e_i$ with $2 \le i \le p$, there exists a
			vertex $y_j$ with $1 \le j < i$ such that $\phiv(e_i) \in \overline{\varphi}(y_j)$.
		\end{itemize}
	\end{Definition}
	
	Let $F=(x,e_1,y_1,\dots,e_p,y_p)$ be a multi-fan at $x$ with respect to $e_1$ and $\phiv$, and $y_{l_1}, y_{l_2},\dots, \allowbreak y_{l_h}$ be a subsequence of $y_2,\dots,y_p$. For any color $\alpha \in [k]$, the $\alpha$-induced sequence with respect to $e_1$ and $\phiv$ satisfies $\phiv(e_{l_1})=\alpha \in \overline{\varphi}(y_1)$ and $\phiv(e_{l_i}) \in \overline{\varphi}(y_{l_{i-1}})$ for every $2\le i\le h$.
	Note that $(x,e_1,y_1,e_{l_1},y_{l_1},\dots,e_{l_h},y_{l_h})$ is still a multi-fan at $x$ with respect to $e_1$ and $\phiv$.
	The colors in $\overline{\varphi}(y_{l_i})$ are  induced by $\alpha$ for every $i \in [h]$.
	
	\begin{Definition}[Kierstead path]
		Let $G$ be a graph with an edge $e_1=y_0y_1$, and an edge coloring $\phiv\in \CC^{k}(G-e_1)$ for some integer $k\ge\D(G)$.	
		A {\em Kierstead} path with respect to $e_1$ and $\phiv$ is a sequence $K=(y_0,e_1,y_1,\dots,e_p,y_p)$ with $p \ge 1$ consisting of edges $e_1, \dots,e_p$ and vertices $y_0,y_1, \dots ,y_p$ satisfying the following two conditions$\colon$
		\begin{itemize}
			\item [{K1.}] The vertices $y_0, \dots, y_p$ are distinct and $e_i=y_{i-1}y_i$ for $1 \le i \le p$.
			\item [{K2.}] For every edge $e_i$ with $2 \le i \le p$, there exists a vertex $y_j$  with $0 \le j < i$ such that $\phiv(e_i) \in \overline{\varphi}(y_j)$.
		\end{itemize}
	\end{Definition}

	\begin{Lemma}[\cite{CCJ2025,SSTF2012}] \label{L2.2}
		Let $G$ be a Class $2$ graph with a critical edge $e_1=xy_1$. Let $\phiv\in\CC^{\D}(G-e_1)$ and $F=(x,e_1,y_1,\dots,e_p,y_p)$ be a multi-fan at $x$ with respect to $e_1$ and $\phiv$. Then the following statements hold.
		\begin{itemize}
			\item [$(1)$] $V(F)$ is $\phiv$-elementary;
			\item [$(2)$] If $\alpha \in \overline{\varphi}(x)$ and $\beta \in
			\overline{\varphi}(y_i)$ for $ 1 \le i \le p$, then $x$ and $y_i$ are
			$(\alpha,\beta)$-linked;
			\item [$(3)$] Let $i,j \in [p]$ be distinct indices with $\delta \in \overline{\varphi}(y_i)$ and $\lambda \in \overline{\varphi}(y_j)$. If $\delta,\lambda$ are induced by distinct colors from $\overline{\varphi}(y_1)$, then $y_i$ and $y_j$ are $(\delta,\lambda)$-linked.
		\end{itemize}
	\end{Lemma}

	\begin{Lemma}[Vizing's Adjacency Lemma \cite{Vizing1965-1}] \label{L2.1}
		Let $G$ be a Class $2$ graph with maximum degree $\D$.
		If $e=xy$ is a critical edge of $G$, then $x$ has at least $max\{\D-d_G(y)+1,2\}$ $\D$-neighbors.	
	\end{Lemma}
	
	It is easy to check the following proposition by Lemma \ref{L2.1}.
	
	\begin{Proposition}\label{P2.3}
		For $4$-critical graphs, the following structural properties hold.
		\begin{itemize}
			
			\item [$(1)$] Every $2$-vertex is adjacent to exactly two $4$-vertices and every $4$-vertex adjacent to a $2$-vertex has all its remaining neighbors being $4$-vertices. Consequently, no two distinct $2$-vertices share a common $4$-neighbor, and a $2$-vertex cannot share a $4$-neighbor with any $3$-vertex;

			\item [$(2)$] Every $4$-vertex has at most two $3$-neighbors; every $3$-vertex has at least two $4$-neighbors; and every $3$-vertex has at most one $3$-neighbor.
		\end{itemize}
	\end{Proposition}

	\begin{Lemma}[\cite{CCS2022}] \label{L2.3}
		Let $G$ be a Class $2$ graph of order $n$ that has a full-deficiency pair $(a,b)$, where the edge $ab$ is a critical edge of $G$. Then $G$ satisfies the following properties.
		\begin{itemize}
			\item [$(1)$] For every $x\in(N_G(a)\cup N_G(b))\backslash\{a,b\}$,
			$d_G(x)=\Delta$;
			\item [$(2)$] For every $x\in V(G)\backslash\{a,b\}$, if $dist_{G}(x,\{a,b\})=2$, then $d_G(x)\ge\Delta-1$. If also $d_G(a)<\Delta$ and
			$d_G(b)<\Delta$, then $d_G(x)=\Delta$.
		\end{itemize}
	\end{Lemma}
	
	We call the structure $H$ a {\it short-branch}, where $V(H)=\{a,b,u,x,y\}$ and $E(H)=\{ab,bu,ux,uy\}$.

	\begin{Lemma}[\cite{Qi-Feng2026}] \label{L2.4}
		Let $G$ be a Class $2$ graph, $H\subseteq G$ be a short-branch with $V(H)=\{a,b,u,x,y\}$ and $E(H)=\{ab,bu,ux,uy\}$, and let $\phiv\in\CC^\Delta(G-ab)$. If
		\[
		\begin{gathered}
			K=(a,ab,b,bu,u,ux,x)\;{and}\;K^*=(a,ab,b,bu,u,uy,y)
		\end{gathered}
		\]
		are Kierstead paths with respect to $ab$ and $\phiv$, then  $|\overline{\varphi}(x)\cap(\overline{\varphi}(a)\cup\overline{\varphi}(b))|+|\overline{\varphi}(y)\cap(\overline{\varphi}(a)\cup\overline{\varphi}(b))|\le 1$. Furthermore, if $(a,b)$ is a full-deficiency pair, then $max\{d_G(x),d_G(y)\}=\Delta(G)$.
	\end{Lemma}

	\begin{Lemma}[\cite{FS1975}] \label{L2.5}
		Let $G$ be a $\Delta$-critical graph of order $n$. Then $n_{\Delta} \ge \frac{2n}{\Delta}$.
	\end{Lemma}

        \begin{Lemma}[\cite{BS1997(1)}] \label{L2.8}
	The smallest $4$-critical graph of even order has  18 vertices.
\end{Lemma}
	
	\begin{Lemma}[\cite{Yap}] \label{L2.6}
		There does not exist a $\Delta$-critical graph with one $2$-vertex $u$, one vertex $v$ of degree less than $\D$ $(v \neq u)$, and all the other vertices being $\Delta$-vertices, where $\Delta \ge 4$.
	\end{Lemma}
	
    \begin {Lemma}[\cite{Yap}] \label{L2.7}
	There does not exist a $4$-critical graph with exactly one $2$-vertex, two $3$-vertices, and all the other vertices being $4$-vertices.
    \end{Lemma}

Mainly using Proposition \ref{P2.3} and Lemmas \ref{L2.3}-\ref{L2.4}, we provide a new proof of the following proposition for $4$-critical graphs whose proof in \cite{Pi} also contains several analogous gaps.

\begin{Proposition}\label{P2.10}
A $4$-critical graph of order $n$ contains at most one $2$-vertex, where $n \in \{5,7,9,11\}$.
\end{Proposition}

\begin{proof}
	Suppose to the contrary that $G$ is a $4$-critical graph of order $n$ with $n_2 \ge 2$, where $n \in \{5,7,9,11\}$. Since $2n_2 \le n_4 \le n-n_2$ by Proposition \ref{P2.3}(1), we get $n_2 \le n/3$. Thus $2 \le n_2 \le 3$. Notice that $\delta(G) \ge 2$ as $G$ is $4$-critical.
	
	Claim: If $n_2=3$, then $n_3\neq 0$.
	
	\begin{proof}
		Suppose  that $n_2=3$ and $n_3= 0$ in $G$.
		By Proposition \ref{P2.3}(1), $n_4 \ge 2n_2=6$, which implies $n \ge 9$.
		Let $v_1, v_2, v_3$ be the three $2$-vertices, and $v_4, \dots, v_n$ be the $4$-vertices of $G$.
		By Proposition \ref{P2.3}(1), let $N_G(v_1)=\{v_4,v_5\}$, $N_G(v_2)=\{v_6,v_7\}$, and $N_G(v_3)=\{v_8,v_9\}$.
       Since $d_G(v_1)+d_G(v_4)=\Delta(G)+2$,
      $(v_1,v_4)$ is a full-deficiency pair.
		By Lemma \ref{L2.3}(2), $v_4$ is not adjacent to
		any vertex among $\{v_6,v_7,v_8,v_9\}$ since $d_G(v_2)=d_G(v_3)=2 < \Delta(G)-1$. When $n=9$, $v_4$ has at most two neighbors $v_1,v_5$, contradicting the fact that $d_G(v_4)=4$.
		When $n=11$,  $v_4$ is adjacent to both $v_{10}$ and $v_{11}$ since $d_G(v_4)=4$. Similarly, each vertex in $\{v_5,\dots,v_9\}$ is adjacent to both $v_{10}$ and $v_{11}$. Hence, $v_{10}$ has at least six neighbors $v_4,\dots,v_9$,
        %$v_4,v_5,v_6,v_7,v_8,v_9$,
		contradicting the fact that $d_G(v_{10})=4$.
	\end{proof}
	
	If $n_2=3$, then $n_4 \ge 2n_2= 6$ by Proposition \ref{P2.3}(1). The claim above implies $n_3 \neq 0$, and since the sum of degrees is even, $n_3$ is even. Since $h \in\{5,7,9,11\}$, we have $n_3 = 2$ and $n_4=6$. By Proposition \ref{P2.3}, counting the minimum  required number of $4$-neighbors of the $2$-vertices and the $3$-vertices, we have $n_4 \ge 2n_2 +2 =8$, contradicting  that $n_4 = 6$.
	
	If $n_2=2$, then $n_4 \ge 2n_2= 4$ by Proposition \ref{P2.3}(1). Note that $n_3 \neq 0$ by Lemma \ref{L2.6} and $n_3$ is even. Thus $n_3 \ge 2$. By Proposition \ref{P2.3}, we have $n_4 \ge 2n_2 +2 =6$. Since $h \in\{5,7,9,11\}$, we have $n_3=2$ and $n_4=7$.
	Let $v_1$, $v_2$ be the two $2$-vertices, $v_3$, $v_4$ be the two $3$-vertices, and $v_5,\dots,v_{11}$ be the $4$-vertices of $G$. By Proposition \ref{P2.3}(1), let $N_G(v_1)=\{v_5,v_6\}$ and $N_G(v_2)=\{v_7,v_8\}$.
	
	First suppose that $v_3v_4 \notin E(G)$. By Proposition \ref{P2.3}(1), there is no common neighbor between $\{v_1,v_2\}$ and $\{v_3,v_4\}$. Thus  $N_G(v_3)=N_G(v_4)=\{v_9,v_{10},v_{11}\}$.
    %Since $d_G(v_1)+d_G(v_5)= \Delta(G)+2$,
    Note that $(v_1,v_5)$ is a full-deficiency pair in $G$. Therefore, $v_5$ is not adjacent to either $v_7$ or $v_8$ by Lemma \ref{L2.3}(2) since $d_G(v_2)=2 <\Delta(G)-1$. Since $d_G(v_5)=4$, $v_5$ has at least two neighbors in $\{v_9,v_{10},v_{11}\}$. Without loss of generality, assume that $v_5v_{9}, v_5v_{10} \in E(G)$. Let $\phiv\in\CC^4(G-v_1v_5)$.  Note that the subgraph $G_1$ is a short-branch with $V(G_1)=\{v_1,v_5,v_9,v_3,v_4\}$ and $E(G_1)=\{v_1v_5,v_5v_9,v_9v_3,v_9v_4\}$. Since $(v_1,v_1v_5,v_5)$ is a multi-fan, $\{v_1,v_5\}$ is $\phiv$-elementary by Lemma \ref{L2.2}(1). Then $|\overline{\varphi}(v_1)\cup \overline{\varphi}(v_5)|=|\overline{\varphi}(v_1)|+|\overline{\varphi}(v_5)|=\Delta(G)=4$, and hence $\overline{\varphi}(v_1)\cup \overline{\varphi}(v_5)=[4]$.
	Thus, in  $G_1$, $K:=(v_1,v_1v_5,v_5,v_5v_9,v_9,v_9v_3,v_3)$ and $K^*:=(v_1,v_1v_5,v_5,v_5v_9,v_9,v_9v_4,v_4)$ are two Kierstead paths. Since $(v_1,v_5)$ is a full-deficiency pair, we have $max\{d_G(v_3),d_G(v_4)\}=\Delta(G)$ by Lemma \ref{L2.4}, which contradicts the fact that $d_G(v_3)=d_G(v_4)=3<\Delta(G)$.
	
	Therefore, $v_3v_4 \in E(G)$.  Since  $(v_1,v_5)$ and  $(v_1,v_6)$ are two full-deficiency pairs and $d_G(v_2)=2< \Delta(G)-1$, there is no edge between $\{v_5,v_6\}$ and $\{v_7,v_8\}$ by Lemma \ref{L2.3}(2). By Proposition \ref{P2.3}(1), there is no common neighbor between $\{v_1,v_2\}$ and $\{v_3,v_4\}$. Thus, all $4$-neighbors of $v_3$ and $v_4$ are contained in $\{v_9,v_{10},v_{11}\}$.
	We consider the following two cases. If $v_3$ and $v_4$ have exactly one common $4$-neighbor, then let $v_9,v_{10} \in N_G(v_3)$ and $v_{10},v_{11} \in N_G(v_4)$.  We claim that $v_5v_{10}\notin E(G)$. Otherwise, $v_5v_{10}\in E(G)$ and thus the subgraph $G_2$ is a short-branch with $V(G_2)=\{v_1,v_5,v_{10},v_3,v_4\}$ and $E(G_2)=\{v_1v_5,v_5v_{10},v_{10}v_3,v_{10}v_4\}$, and the same argument as $G_1$ yields a contradiction to Lemma \ref{L2.4}. Since $d_G(v_5)=4$, $v_5$ is adjacent to both $v_9$ and $v_{11}$. The same holds for $v_6,v_7,v_8$. Therefore, $v_9$ has at least five neighbors $v_3,v_5,v_6,v_7,v_8$, contradicting the fact that $d_G(v_9)=4$.
	If $v_3$ and $v_4$ have exactly two common $4$-neighbors, then let $v_9,v_{10} \in N_G(v_3)\cap N_G(v_4)$.
We claim that $v_5v_9 \notin E(G)$; otherwise, $v_5v_9 \in E(G)$ and thus the subgraph $G_3$ is a short-branch with $V(G_3)=\{v_1,v_5,v_9,v_3,v_4\}$ and $E(G_3)=\{v_1v_5,v_5v_9,v_9v_3,v_9v_4\}$, and the same argument as $G_1$ yields a contradiction to Lemma \ref{L2.4}. Similarly, $v_5v_{10} \notin E(G)$. Hence, $v_5$ has at most three neighbors $v_1,v_6,v_{11}$, contradicting the fact that $d_G(v_5)=4$. This completes the proof.
\end{proof}

In light of the following Propositions \ref{P2.11}-\ref{P2.12} from \cite{CR2019},
%with a structural analysis based %on forbidden configurations
we obtain four new structural Propositions \ref{P2.13}-\ref{P2.16} of $4$-critical graphs, which are instrumental in proving Theorem \ref{T1}.  None of these configurations in Figure $1$ can appear in any $4$-critical graph.  The proofs of Propositions \ref{P2.13}-\ref{P2.16} will be given in the Section $4$ of this paper.

\begin{figure}[H]
	\begin{center}
		\centering
		\scalebox{0.27}{\includegraphics{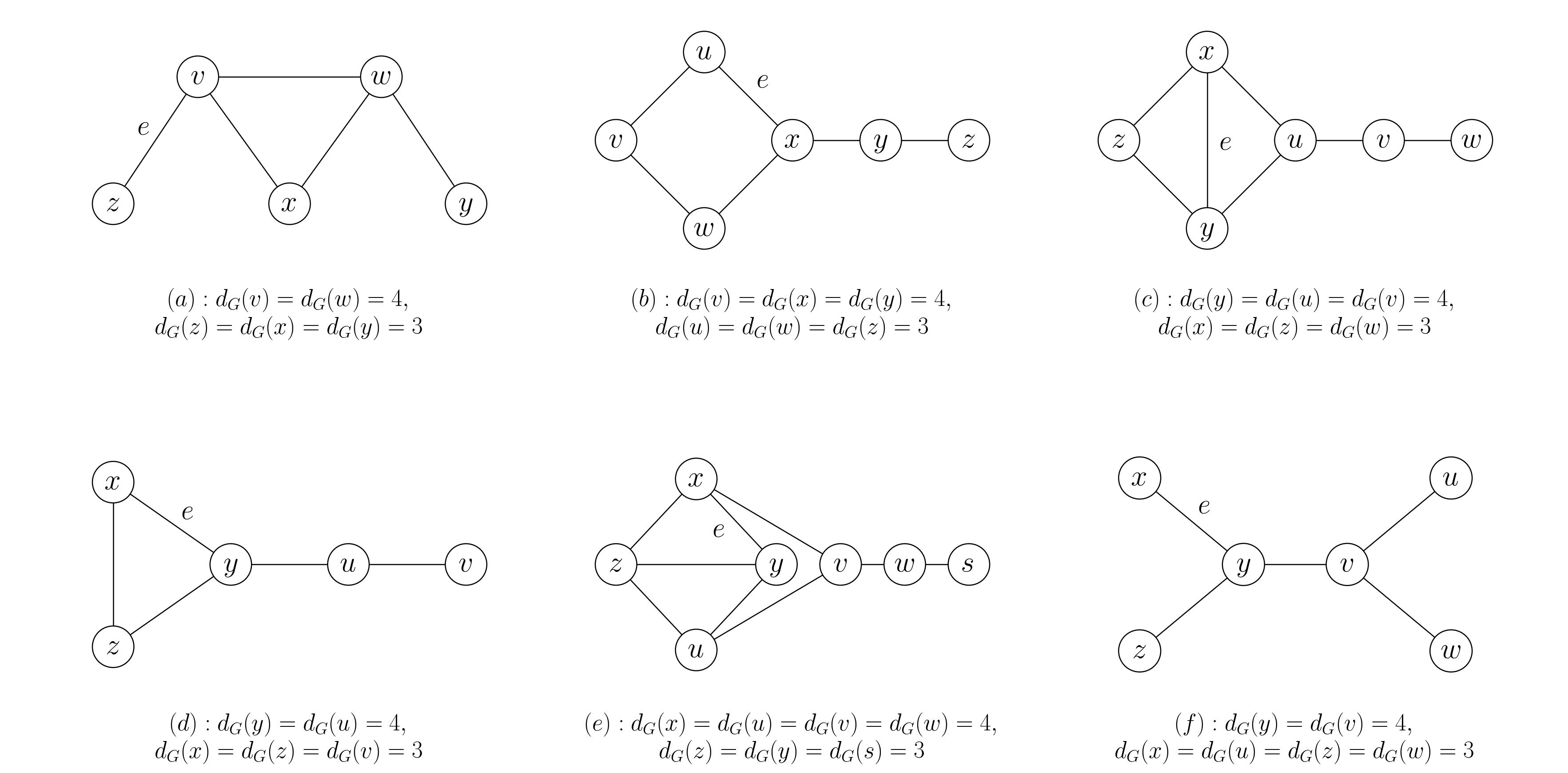}}
		\captionsetup{labelsep=none}
		\caption{}
	\end{center}
\end{figure}

\begin{Proposition}[\cite{CR2019}]\label{P2.11}
	Suppose that $G$ is a graph with $\Delta(G)=4$ and contains the configuration in $Figure$ 1(a).
	If $\chi'(G-e)=4$, then $\chi'(G)=4$.
\end{Proposition}

\begin{Proposition}[\cite{CR2019}]\label{P2.12}
	Suppose that $G$ is a graph with $\Delta(G)=4$ and contains the configuration in $Figure$ 1(b).
	If $\chi'(G-e)=4$, then $\chi'(G)=4$.
\end{Proposition}

\begin{Proposition}\label{P2.13}
	Suppose that $G$ is a graph with $\Delta(G)=4$ and contains the configuration in $Figure$ 1(c).
	If $\chi'(G-e)=4$, then $\chi'(G)=4$.
\end{Proposition}

\begin{Proposition}\label{P2.14}
	Suppose that $G$ is a graph with $\Delta(G)=4$ and contains the configuration in $Figure$ 1(d).
	If $\chi'(G-e)=4$, then $\chi'(G)=4$.
\end{Proposition}

\begin{Proposition}\label{P2.15}
	Suppose that $G$ is a graph with $\Delta(G)=4$ and contains the configuration in $Figure$ 1(e). If $\chi'(G-e)=4$, then $\chi'(G)=4$.
\end{Proposition}

\begin{Proposition}\label{P2.16}
	Suppose that $G$ is a graph with $\Delta(G)=4$ and contains the configuration in $Figure$ 1(f). If $\chi'(G-e)=4$, then $\chi'(G)=4$.
\end{Proposition}

\section{Proof of Theorem 1}

\begin{proof}
Sufficiency is easy. Indeed, if $G$ contains an overfull subgraph $H$ with $\Delta(H)=\Delta(G)=4$, then $H$ needs at least $5$ colors to get a proper edge-coloring. Hence $G$ also needs at least $5$ colors, and by Vizing's theorem $\chi'(G)\le \Delta(G)+1=5$, so $G$ is Class 2.

Now we will prove necessity. Since $G$ is a Class $2$ graph with $\Delta(G)=4$, $G$ contains a $4$-critical subgraph $H$ of order $h$. Note that $\D(G)=4>n/3$. Hence $h\le 11$. By Lemma \ref{L2.8}, no even-order $4$-critical graph has order at most $10$, hence $h\in\{5,7,9,11\}$. If $H$ is a $4$-regular graph, then $|E(H)|=2h > 4\left\lfloor h/2 \right\rfloor$, which implies that $H$ is overfull. Now $G$ contains an overfull subgraph $H$ with $\Delta(H)=4$, as desired. Therefore, it remains to consider the case that $H$ is not $4$-regular.  Note that $\delta(H)\ge2$ as $H$ is critical. Let $h_i$ denote the number of $i$-vertices in $H$ with $i\in \{2,3,4\}$.
Since $h\in\{5,7,9,11\}$ and $\lceil h/2 \rceil \le h_4$ by Lemma \ref{L2.5}, we have $h-5 \le \lceil h/2 \rceil \le h_4 \le h-1 $. Notice that $h_3$ is even since the sum of degrees is even.
	
	Claim $1$: The case $h_2=1$, $h_3=4$, and $h_4=h-5$ cannot hold.
	
	\begin{proof}
		
	Suppose that $h_2=1$, $h_3=4$, and $h_4=h-5$. By Proposition \ref{P2.3}, counting the minimum  required number of  $4$-neighbors of the $2$-vertices and the $3$-vertices, we have $h_4 \ge 2h_2 +4 =6$. Since $h\in\{5,7,9,11\}$, it follows that $h_4=6$ and $h=11$.
	Let $v_1$ be the $2$-vertex, $v_2,v_3,v_4,v_5$ be the $3$-vertices, and $v_6,\dots,v_{11}$ be the $4$-vertices of $H$. By Proposition \ref{P2.3}(1), let $N_H(v_1)=\{v_6,v_7\}$, and  there is no edge between $\{v_2,v_3,v_4,v_5\}$ and $\{v_6,v_7\}$. Therefore,
    $v_6$ has at least two neighbors in $\{v_8,v_9,v_{10},v_{11}\}$  since $d_H(v_6)=4$.
	Without loss of generality, assume that $v_6v_8, v_6v_9 \in E(H)$.
	By Proposition \ref{P2.3}(2), so each vertex in $\{v_2,v_3,v_4,v_5\}$ has at least two neighbors in $\{v_8,v_9,v_{10},v_{11}\}$, which gives at least eight edges between the two sets. On the other hand, each vertex in $\{v_8,v_9,v_{10},v_{11}\}$ has at most two neighbors in $\{v_2,v_3,v_4,v_5\}$, which gives at most eight edges between the two sets. Therefore, exactly eight edges exist between the two sets, and each vertex in either set is adjacent to exactly two vertices in the other set.
	Without loss of generality, $v_8v_2,v_8v_3\in E(H)$. Let $\phiv\in\CC^4(H-v_1v_6)$.  Note that $H_1$ is a short-branch with $V(H_1)=\{v_1,v_6,v_8,v_2,v_3\}$ and $E(H_1)=\{v_1v_6,v_6v_8,v_8v_2,v_8v_3\}$. Since $(v_1,v_1v_6,v_6)$ is a multi-fan, $\{v_1,v_6\}$ is $\phiv$-elementary by Lemma \ref{L2.2}(1). Thus $|\overline{\varphi}(v_1)\cup \overline{\varphi}(v_6)|=|\overline{\varphi}(v_1)|+|\overline{\varphi}(v_6)|=\Delta(H)=4$ and $\overline{\varphi}(v_1)\cup \overline{\varphi}(v_6)=[4]$. Therefore, in $H_1$,  $K:=(v_1,v_1v_6,v_6,v_6v_8,v_8,v_8v_2,v_2)$ and $K^*:=(v_1,v_1v_6,v_6,v_6v_8,v_8,v_8v_3,v_3)$ are two Kierstead paths. Since $d_H(v_1)+d_H(v_6)= \Delta(H)+2$, $(v_1,v_6)$ is a full-deficiency pair in $H$. Thus we have $max\{d_H(v_2),d_H(v_3)\}=\Delta(H)$ by Lemma \ref{L2.4}, which contradicts the fact that $d_H(v_2)=d_H(v_3)=3 < \Delta(H)$.	
	\end{proof}

	By Lemma \ref{L2.7}, Proposition \ref{P2.10}, and Claim $1$, there are still three cases for $H$: $h_2=1$, $h_3=0$, and $h_4=h-1$; $h_2=0$, $h_3=2$, and $h_4=h-2$; or $h_2=0$, $h_3=4$, and $h_4=h-4$. For the first two cases,  $|E(H)|=2h-1 > 4 \lfloor h/2 \rfloor=2(h-1)$, so $H$ is overfull. Now $G$ contains an overfull subgraph $H$ with $\D(H)=4$, as desired. Thus, only the last case for $H$ requires further analysis, which will  be shown as impossible, and then proceed to complete the entire proof.

    By Proposition \ref{P2.3}(2),  counting the minimum  required number of $4$-neighbors of the $3$-vertices, we have $h_4 \ge 4$. As $h \in \{5,7,9,11\}$, it follows that  $h = 9$ or $h=11$. Let $v_1,v_2,v_3,v_4$ be the four $3$-vertices and $v_5, \dots, v_h$ be the $4$-vertices of $H$. Since $H$ is critical, every edge of $H$ is critical; hence any configuration in Figure $1$ is forbidden in $H$.

    Observation: If two adjacent $3$-vertices have two common 4-neighbors in $H$, then the two common 4-neighbors are not adjacent.
    \begin{proof}
    Let $v_1v_2 \in E(H)$ and $v_5,v_6 \in N_H(v_1)\cap N_H(v_2)$. Suppose that $v_5v_6 \in E(H)$. Let $\phiv\in\CC^{4}(H-v_1v_2)$. Since $(v_1,v_1v_2,v_2)$ is a multi-fan, $\{v_1,v_2\}$ is $\phiv$-elementary by Lemma \ref{L2.2}(1). Let $\overline{\varphi}(v_1)=\{1,2\}$ and $\overline{\varphi}(v_2)=\{3,4\}$. Then color edges $v_1v_5$ and $v_1v_6$ with colors $3$ and $4$, respectively, and $v_2v_5$ and $v_2v_6$ with colors $1$ and $2$, respectively. Thus the edge $v_5v_6$ cannot be colored from $\{1,2,3,4\}$, which is a contradiction.
    \end{proof}
	
	Claim $2$: $h=11$.
	
	\begin{proof}
	Suppose that $h=9$. By Proposition \ref{P2.3}(2), we have  $|E(H[\{v_1,v_2,v_3,v_4\}])| \le 2$. We consider the following three cases.
		
		If $|E(H[\{v_1,v_2,v_3,v_4\}])|=0$, then each vertex in $\{v_1,v_2,v_3,v_4\}$ has three neighbors in $\{v_5,\dots,v_9\}$, which gives  twelve edges between the two sets. On the other hand, each vertex in $\{v_5,\dots,v_9\}$ has at most two neighbors in $\{v_1,v_2,v_3,v_4\}$ by Proposition \ref{P2.3}(2), which gives at most ten edges between the two sets, a contradiction.
		
		If $|E(H[\{v_1,v_2,v_3,v_4\}])|=1$, then let $v_1v_2 \in E(H)$ and $v_3v_4 \notin E(H)$. Since $d_H(v_1)+d_H(v_2)=3+3=\Delta(H)+2$, $(v_1,v_2)$ is a full-deficiency pair in $H$. As $d_H(v_1)=d_H(v_2)=3 < \Delta(H)$, there is no common neighbor between $\{v_1,v_2\}$ and $\{v_3,v_4\}$ by Lemma \ref{L2.3}(2), since otherwise   $d_H(v_3)$ or $d_H(v_4)= \Delta(H)$, a contracdition. Therefore, let $N_H(v_1) \cap N_H(v_2)=\{v_5,v_6\}$ and $N_H(v_3) \cap N_H(v_4)=\{v_7,v_8,v_9\}$.
		We claim that $v_5$ is not adjacent to any vertex among $\{v_7, v_8, v_9\}$. Otherwise, suppose that $v_5v_7 \in E(H)$; then $H$ contains  the configuration in Figure 1(b) with
		$(x,y,z,u,v,w):=(v_5,v_7,v_3,v_1,v_6,v_2)$, contradicting {Proposition} \ref{P2.12}. The other cases are similar. By the Observation, $v_5$ is not adjacent to $v_6$. Hence, $v_5$ has two neighbors $v_1,v_2$, contradicting the fact that $d_H(v_5)=4$.
		
		If $|E(H[\{v_1,v_2,v_3,v_4\}])|=2$, then let $v_1v_2,v_3v_4 \in E(H)$. Since $d_H(v_1)+d_H(v_2)=\Delta(H)+2$ and $d_H(v_1)=d_H(v_2)=3 < \Delta(H)$, there is no common neighbor between $\{v_1,v_2\}$ and $\{v_3,v_4\}$  by Lemma \ref{L2.3}(2).
        %, since otherwise $d_H(v_3)$ or $d_H(v_4)=\Delta(H)$, a contradition.
        Since there are only five 4-vertices available, two adjacent 3-vertices must have at least one common neighbor in order to satisfy their degree requirements. Now we consider the number of common neighbors of $v_1$ and $v_2$. When $|N_H(v_1) \cap N_H(v_2)|=1$, let $N_H(v_1)=\{v_2,v_5,v_6\}$ and $N_H(v_2)=\{v_1,v_5,v_9\}$. Since $\{v_1,v_2\}$ and $\{v_3,v_4\}$ have no common neighbors between them, $N_H(v_3) \cap N_H(v_4)=\{v_7,v_8\}$.
		We claim that $v_6$ is not adjacent to either $v_7$ or $v_8$.
		Otherwise, suppose that $v_6v_7 \in E(H)$; then $H$ contains the configuration in Figure 1(b) with
		$(x,y,z,u,v,w):=(v_7,v_6,v_1,v_3,v_8,v_4)$, contradicting {Proposition} \ref{P2.12}.  The other case is similar.
		Hence, $v_6$ has at most three neighbors $v_1,v_5,v_9$, contradicting the fact that $d_H(v_6)=4$.
		When $|N_H(v_1) \cap N_H(v_2)|=2$, let $N_H(v_1) \cap N_H(v_2)=\{v_5,v_6\}$. If $|N_H(v_3) \cap N_H(v_4)|=1$, then this reduces to the case $|N_H(v_1) \cap N_H(v_2)|=1$. If $|N_H(v_3) \cap N_H(v_4)|=2$, then let $N_H(v_3) \cap N_H(v_4)=\{v_7,v_8\}$.
		We claim that $v_5$ is not adjacent to either $v_7$ or $v_8$.		
		Otherwise, suppose that $v_5v_7 \in E(H)$; then $H$ contains the configuration in Figure 1(b) with
		$(x,y,z,u,v,w):=(v_5,v_7,v_3,v_1,v_6,v_2)$, contradicting {Proposition} \ref{P2.12}. The other case is similar.
		By the Observation, we have $v_5v_6 \notin E(H)$.
		Hence, $v_5$ has at most three neighbors $v_1,v_2,v_9$, contradicting  the fact  that $d_H(v_5)=4$. This completes the proof.
	\end{proof}

	By Proposition \ref{P2.3}(2), we have $|E(H[\{v_1,v_2,v_3,v_4\}])| \le 2$. We consider the following three cases to get contradictions for $H$.

	\begin{figure}[htbp]
\centering
\includegraphics[width=\linewidth]{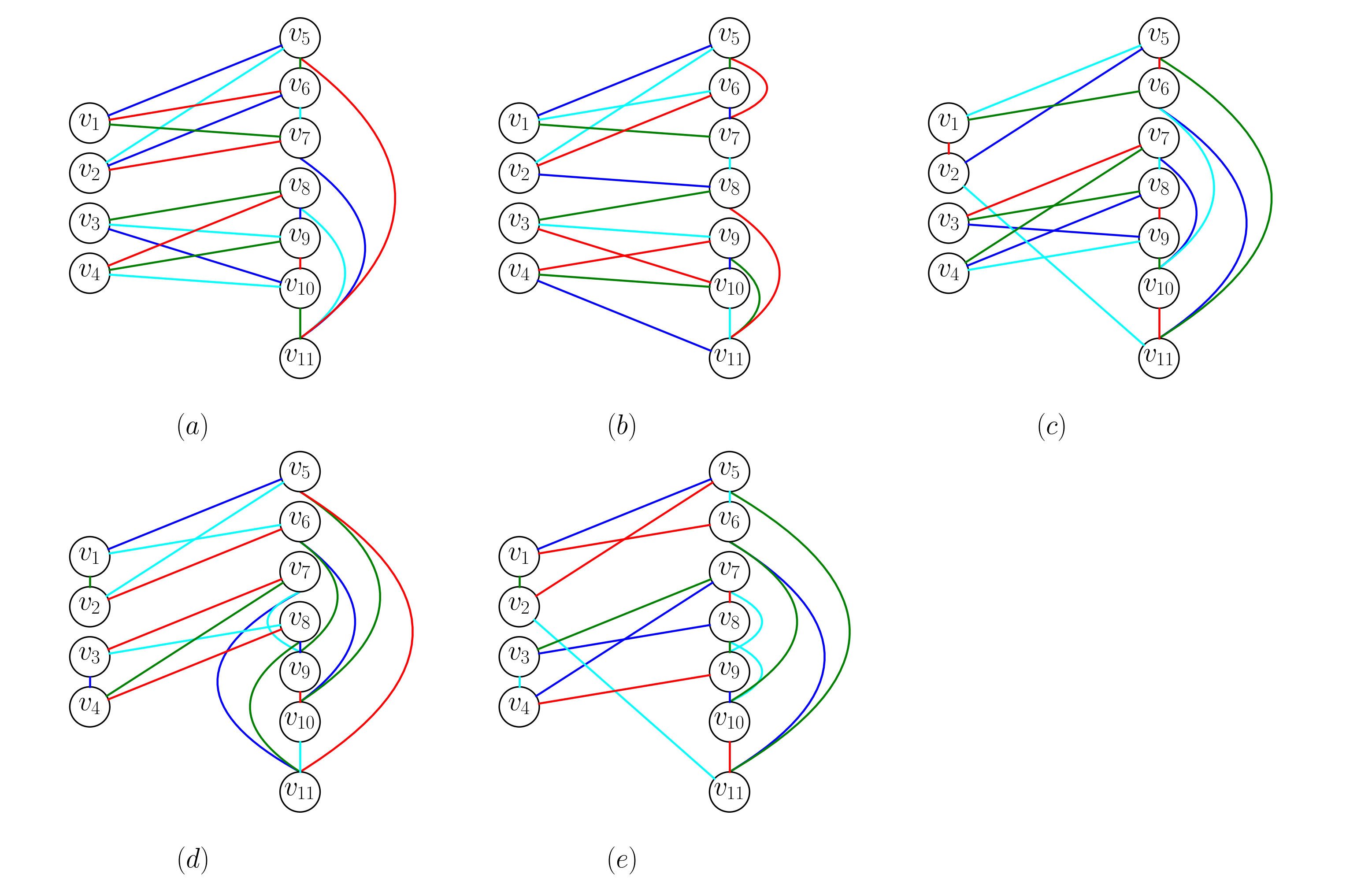}
\captionsetup{labelsep=none}
\caption{}
\end{figure}

\begin{figure}[htbp]
\centering
\includegraphics[width=\linewidth]{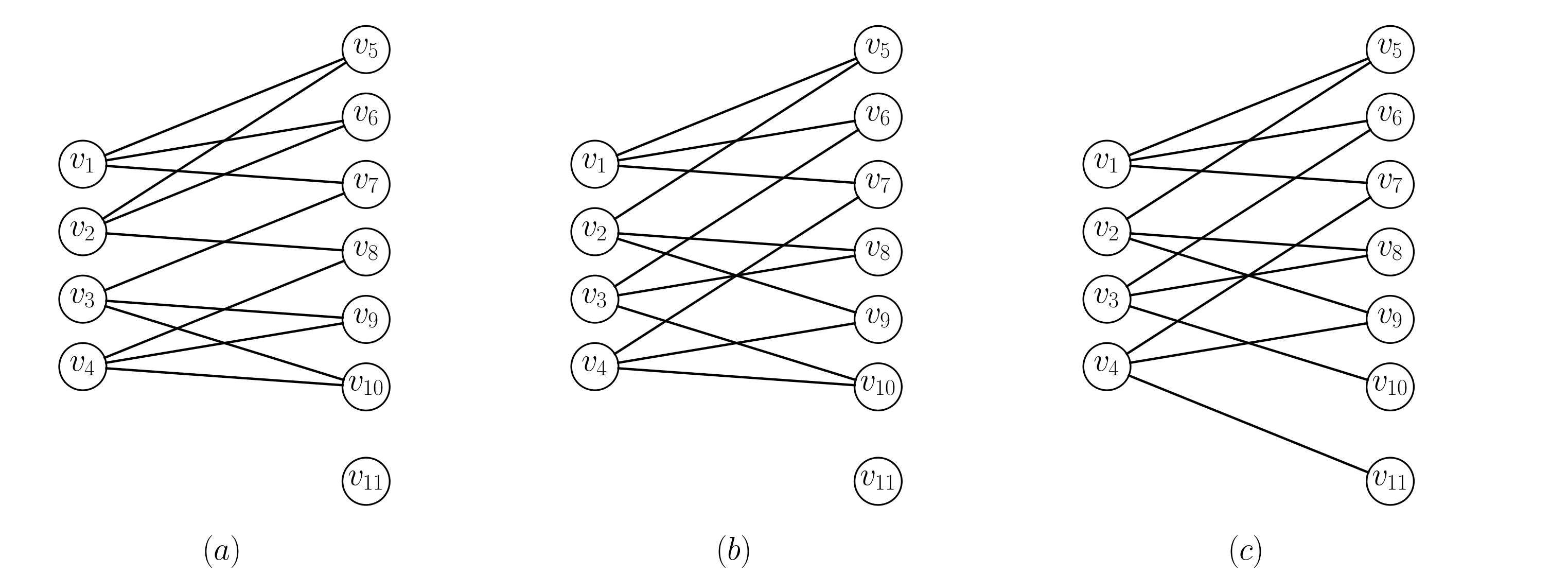}
\captionsetup{labelsep=none}
\caption{}
\end{figure}

	Case $1$. $|E(H[\{v_1,v_2,v_3,v_4\}])|=0$.
	
	Let $N_H(v_1)=\{v_5,v_6,v_7\}$. Each vertex in $\{v_1,v_2,v_3,v_4\}$ has exactly three neighbors in $\{v_5,\dots, v_{11}\}$, which gives twelve edges between the two sets.  We consider the following two subcases by Proposition \ref{P2.3}(2).

	Subcase $1.1$. Six $4$-vertices have two $3$-neighbors.

	First suppose that two $4$-vertices have two common $3$-neighbors. Without loss of generality, let $v_5,v_6 \in N_H(v_2)$. If $v_2v_7 \in E(H)$, then let $N_H(v_3)=N_H(v_4)=\{v_8,v_9,v_{10}\}$. There is no edge between $\{v_5,v_6,v_7\}$ and $\{v_8,v_9,v_{10}\}$. Otherwise, suppose that $v_5v_8 \in E(H)$; then $H$  contains the configuration in Figure 1(b) with
	$(x,y,z,u,v,w):=(v_5,v_8,v_3,v_1,v_6,v_2)$, contradicting  {Proposition} \ref{P2.12}. The other cases are similar.
	Therefore, each of $\{v_5,v_6,v_7\}$ and $\{v_8,v_9,v_{10}\}$ contains two edges. Without loss of generality,  $H$ is precisely the graph shown in Figure $2$(a), which has a proper $4$-edge-coloring, contradicting the fact that $H$ is a $4$-critical graph.  If $v_2v_7 \notin E(H)$, then without loss of generality, $H$ contains the local structure shown in Figure $3$(a). We claim that $v_7$ is not adjacent to any vertex among $\{v_5, v_6, v_9, v_{10}\}$. Otherwise, suppose that $v_5v_7 \in E(H)$; then $H$  contains the configuration in Figure 1(a) with
	$(x,y,z,v,w):=(v_1,v_3,v_2,v_5,v_7)$, contradicting {Proposition} \ref{P2.11}. The other cases are similar.
	We also have that $v_7$ is not adjacent to $v_8$. Otherwise, suppose that $v_7v_8 \in E(H)$; then $H$  contains the configuration in Figure 1(f) with
	$(x,y,z,u,v,w):=(v_1,v_7,v_3,v_2,v_8,v_4)$, contradicting {Proposition} \ref{P2.16}.
	Therefore, $v_7$ has at most three neighbors $v_1,v_3,v_{11}$, contradicting  the fact that $d_H(v_7)=4$.

	Therefore, no two $4$-vertices have two common $3$-neighbors. Then $H$ has the local structure shown in Figure 3(b). We claim that $v_5$ is not adjacent to any vertex among $\{v_6, v_7, v_8, v_9\}$. Otherwise, suppose that $v_5v_6 \in E(H)$; then $H$  contains the configuration in Figure 1(a) with
	$(x,y,z,v,w):=(v_1,v_3,v_2,v_5,v_6)$, contradicting {Proposition} \ref{P2.11}. The other cases are similar.  We also have that $v_5$ is not adjacent to $v_{10}$. Otherwise, suppose that $v_5v_{10} \in E(H)$; then $H$  contains the configuration in Figure 1(f) with
	$(x,y,z,u,v,w):=(v_1,v_5,v_2,v_3,v_{10},v_4)$, contradicting {Proposition} \ref{P2.16}.
	Thus, $v_5$ has at most three neighbors $v_1,v_2,v_{11}$, contradicting the fact  that $d_H(v_5)=4$.
	
	Subcase $1.2$. Five $4$-vertices have two $3$-neighbors, and two $4$-vertices have one $3$-neighbor.

		First suppose that two $4$-vertices have two common $3$-neighbors. Without loss of generality, let $v_5,v_6\in N_H(v_2)$. If $v_2v_7\in E(H)$, then each vertex in $\{v_8, \dots, v_{11}\}$ has at least one neighbor in $\{v_3, v_4\}$. There is no edge between $\{v_5,v_6,v_7\}$ and $\{v_8,v_9,v_{10},v_{11}\}$; otherwise, the configuration in Figure $1$(b) would arise, contradicting {Proposition} \ref{P2.12}. Thus, $v_5v_6,v_5v_7,v_6v_7 \in E(H)$.
	Let $H_2 = H[\{v_1, v_2, v_5, v_6, v_7\}]$. Note that $d_{H_2}(v)=d_H(v)$ for every $v\in V(H_2)$.  There is no edge between $V(H_2)$ and $V(H)\setminus V(H_2)$,  contradicting the fact that $H$ is connected.
	If $v_2v_7\notin E(H)$, then let $v_2v_8 \in E(H)$.  Each vertex in $\{v_9, v_{10}, v_{11}\}$ has at least one neighbor in $\{v_3, v_4\}$ and thus we claim that $v_5$ is not adjacent to any vertex among $\{v_9,v_{10},v_{11}\}$; otherwise, the configuration in Figure $1$(b) would arise, contradicting  {Proposition} \ref{P2.12}. So $v_5$ is adjacent to exactly two vertices in $\{v_6,v_7,v_8\}$. When $v_5v_7,v_5v_8\in E(H)$,  there is no edge between $\{v_3,v_4\}$ and $\{v_7,v_8\}$. Otherwise, suppose that $v_3v_7 \in E(H)$; then $H$ contains the configuration in Figure 1(b) with
	$(x,y,z,u,v,w):=(v_5,v_7,v_3,v_1,v_6,v_2)$, contradicting {Proposition} \ref{P2.12}.  The other cases are similar.
	Thus, we have $N_H(v_3)=N_H(v_4)=\{v_9,v_{10},v_{11}\}$. There is no edge between $\{v_9,v_{10},v_{11}\}$ and $\{v_6,v_7,v_8\}$. Otherwise, suppose that $v_6v_9 \in E(H)$; then $H$ contains the configuration in Figure 1(b) with
	$(x,y,z,u,v,w):=(v_9,v_6,v_1,v_3,v_{10},v_4)$, contradicting {Proposition} \ref{P2.12}.  The other cases are similar.
	Thus, $v_9v_{10},v_9v_{11},v_{10}v_{11}\in E(H)$.
    Let $H_3 = H[\{v_3,v_4,v_9,v_{10},v_{11}\}]$. Note that $d_{H_3}(v)=d_H(v)$ for every $v\in V(H_3)$.  There is no edge between $V(H_3)$ and $V(H)\setminus V(H_3)$,  contradicting the fact that $H$ is connected.
	When $v_5v_6,v_5v_7\in E(H)$ (by symmetry), $v_7$ is not adjacent to either $v_3$ or $v_4$. Otherwise, suppose that $v_3v_7 \in E(H)$; then $H$  contains the configuration in Figure 1(b) with
	$(x,y,z,u,v,w):=(v_5,v_7,v_3,v_1,v_6,v_2)$, contradicting  {Proposition} \ref{P2.12}.  The other case is similar. Therefore, all neighbors of $v_3$ and $v_4$ are contained in $\{v_8,v_9,v_{10},v_{11}\}$. We consider the following two subcases. If $|N_H(v_3)\cap N_H(v_4)|=2$, then let $N_H(v_3)=\{v_8,v_9,v_{10}\}$ and $N_H(v_4)=\{v_9,v_{10},v_{11}\}$. We claim that there is no edge between $\{v_9,v_{10}\}$ and $\{v_6,v_7,v_8\}$. Otherwise, suppose that $v_7v_9 \in E(H)$; then $H$  contains the configuration in Figure 1(b) with
	$(x,y,z,u,v,w):=(v_9,v_7,v_1,v_3,v_{10},v_4)$, contradicting {Proposition} \ref{P2.12}.  The other cases are similar. Similarly, we have that $v_6$ is not adjacent to $v_8$.
	Since all remaining adjacencies are uniquely determined by the degree conditions, $H$ is precisely the graph shown in Figure 2(b), which has a proper $4$-edge-coloring, contradicting the fact that $H$ is a $4$-critical graph.
	If $|N_H(v_3)\cap N_H(v_4)|=3$, then $N_H(v_3)= N_H(v_4)=\{v_9,v_{10},v_{11}\}$. Similarly, we  have that there is no edge between $\{v_9,v_{10},v_{11}\}$ and $\{v_6,v_7,v_8\}$. Thus, $v_8$ has at most three neighbors $v_2,v_6,v_7$, contradicting  the fact that $d_H(v_8)=4$.
	
	Therefore, no two $4$-vertices have two common $3$-neighbors. Without loss of generality, $H$ contains the local structure shown in Figure 3(c). We claim that $v_6$ is not adjacent to any vertex among $\{v_5, v_7, v_8\}$. Otherwise, suppose that $v_5v_6 \in E(H)$; then $H$  contains the configuration in Figure 1(a) with
	$(x,y,z,v,w):=(v_1,v_3,v_2,v_5,v_6)$, contradicting  {Proposition} \ref{P2.11}.  The other cases are similar. We also have that $v_6$ is not adjacent to $v_9$. Otherwise, suppose that $v_6v_9 \in E(H)$; then $H$  contains the configuration in Figure 1(f) with
	$(x,y,z,u,v,w):=(v_1,v_6,v_3,v_2,v_9,v_4)$, contradicting  {Proposition} \ref{P2.16}.
	Since $d_H(v_6)=4$, $v_6$ is adjacent to both $v_{10}$ and $v_{11}$. Similarly, each vertex in $\{v_5,v_7,v_8,v_9\}$ is adjacent to both $v_{10}$ and $v_{11}$. Thus, $v_{10}$ has at least six neighbors $v_3,v_5,v_6,v_7,v_8,v_9$, contradicting  the fact that $d_H(v_{10})=4$.

	Case $2$. $|E(H[\{v_1,v_2,v_3,v_4\}])|=1$. Without loss of generality, we may assume that $v_1v_2\in E(H),v_3v_4\notin E(H)$, and $v_5,v_6 \in N_H(v_1)$.
	
	Since $d_H(v_1)+d_H(v_2)=\Delta(H)+2$, $(v_1,v_2)$ is a full-deficiency pair. As $d_H(v_1)=d_H(v_2)=3 < \Delta(H)$, there is no common neighbor between $\{v_1,v_2\}$ and $\{v_3,v_4\}$  by Lemma \ref{L2.3}(2).
	
	Subcase $2.1$. $|N_H(v_1) \cap N_H(v_2)|=2$.
	
	Let $N_H(v_2)=\{v_1,v_5,v_6\}$. Therefore, all neighbors of $v_3$ and $v_4$ are contained in $\{v_7,\dots,\\v_{11}\}$. Let $N_H(v_3)=\{v_7,v_8,v_9\}$.
	By the Observation, we have $v_5v_6 \notin E(H)$.
	
		If $|N_H(v_3) \cap N_H(v_4)|=3$, then $N_H(v_4)=\{v_7,v_8,v_9\}$. We claim that there is no edge between $\{v_5,v_6\}$ and $\{v_7, v_8, v_9\}$. Otherwise, suppose that $v_5v_7 \in E(H)$; then $H$  contains the configuration in Figure 1(b) with
	$(x,y,z,u,v,w):=(v_5,v_7,v_3,v_1,v_6,v_2)$, contradicting {Proposition} \ref{P2.12}.  The other cases are similar.
	Thus,  $v_5$ and $v_6$ are both adjacent to each of $v_{10}$ and $v_{11}$. Note that $v_{10}$ is not adjacent to any vertex among $\{v_7, v_8, v_9\}$. Otherwise, suppose that $v_7v_{10} \in E(H)$; then $H$  contains the configuration in Figure 1(e) with
	$(x,y,z,u,v,w,s):=(v_5,v_2,v_1,v_6,v_{10},v_7,v_3)$, contradicting  {Proposition} \ref{P2.15}.  The other cases are similar.
	Hence, $v_{10}$ has at most three neighbors $v_5,v_6,v_{11}$, contradicting the fact  that $d_H(v_{10})=4$.
	
	If $|N_H(v_3)\cap N_H(v_4)|=1$ ($2$, respectively), then let $N_H(v_4)=\{v_7,v_{10},v_{11}\}$ ($N_H(v_4)=\{v_7,v_8,v_{10}\}$, respectively). We claim that $v_5$ is not adjacent to any vertex among $\{v_7,\dots,v_{11}\}$ ($\{v_7,\dots,v_{10}\}$, respectively). Otherwise, suppose that $v_5v_7 \in E(H)$; then $H$  contains the configuration in Figure 1(b) with
	$(x,y,z,u,v,w):=(v_5,v_7,v_3,v_1,v_6,v_2)$, contradicting {Proposition} \ref{P2.12}.  The other cases are similar.
	In both cases, $v_5$ has at most three neighbors $v_1,v_2,v_{11}$, contradicting the fact  that $d_H(v_5)=4$.

	Subcase $2.2$. $|N_H(v_1) \cap N_H(v_2)|=1$.
	
	Let $ N_H(v_2)=\{v_1,v_5,v_{11}\}$. Therefore, all neighbors of $v_3$ and $v_4$ are contained in $\{v_7,\dots,v_{10}\}$. Let $N_H(v_3)=\{v_7,v_8,v_9\}$.
	
	If $|N_H(v_3) \cap N_H(v_4)|=3$, then $N_H(v_4)=\{v_7,v_8,v_9\}$. There is no edge between $\{v_7,v_8,v_9\}$ and $\{v_5,v_6,v_{11}\}$. Otherwise, suppose that $v_5v_7 \in E(H)$; then $H$  contains the configuration in Figure 1(b) with
	$(x,y,z,u,v,w):=(v_7,v_5,v_1,v_3,v_8,v_4)$, contradicting {Proposition} \ref{P2.12}.  The other cases are similar.
	Thus, the subgraph induced by $\{v_7, v_8, v_9\}$ contains at least two edges. If the subgraph induced by $\{v_7, v_8, v_9\}$ contains two edges, then we may assume without loss of generality that $v_7v_8,v_8v_9 \in E(H)$. Since all remaining adjacencies are uniquely determined by the degree conditions, $H$ is precisely the graph shown in Figure 2(c), which has a proper $4$-edge-coloring, contradicting the fact that $H$ is a $4$-critical graph. If the subgraph induced by $\{v_7, v_8, v_9\}$ contains three edges, then $v_7v_8,v_7v_9,v_8v_9 \in E(H)$. Now $v_{10}$ has at most three neighbors $v_5,v_6,v_{11}$, contradicting the fact that $d_H(v_{10})=4$.
	
	If $|N_H(v_3) \cap N_H(v_4)|=2$, then let $N_H(v_4)=\{v_7,v_8,v_{10}\}$. We claim that $v_5$ is not adjacent to any vertex among $\{v_7,\dots,v_{10}\}$. Otherwise, suppose that $v_5v_7 \in E(H)$; then $H$  contains the configuration in Figure 1(d) with
	$(x,y,z,u,v):=(v_1,v_5,v_2,v_7,v_3)$, contradicting {Proposition} \ref{P2.14}.  The other cases are similar.
	Since $d_H(v_5)=4$, we have $v_5v_6,v_5v_{11} \in E(H)$. We also have that $v_6$ is not adjacent to any vertex among $\{v_7,\dots,v_{10}\}$. Otherwise, suppose that $v_6v_7 \in E(H)$; then $H$  contains the configuration in Figure 1(c) with
	$(x,y,z,u,v,w):=(v_1,v_5,v_2,v_6,v_7,v_3)$, contradicting  {Proposition} \ref{P2.13}.  The other cases are similar.
	Therefore, $v_6$ has at most three neighbors $v_1,v_5,v_{11}$, contradicting the fact  that $d_H(v_6)=4$.

	Subcase $2.3$. $|N_H(v_1) \cap N_H(v_2)|=0$.
	
	Let $ N_H(v_2)=\{v_1,v_{10},v_{11}\}$. Therefore, all neighbors of $v_3$ and $v_4$ are contained in $\{v_7,v_8,v_9\}$, that is, $N_H(v_3)=N_H(v_4)=\{v_7,v_8,v_9\}$. There is no edge between $\{v_7,v_8,v_9\}$ and $\{v_5,v_6,v_{10},v_{11}\}$. Otherwise, suppose that $v_5v_7 \in E(H)$; then $H$  contains the configuration in Figure 1(b) with
	$(x,y,z,u,v,w):=(v_7,v_5,v_1,v_3,v_8,v_4)$, contradicting {Proposition} \ref{P2.12}.  The other cases are similar.
	Thus, $v_7v_8,v_7v_9,v_8v_9 \in E(H)$. Let $H_4 = H[\{v_3,v_4,v_7,v_8,\allowbreak v_9\}]$. Note that $d_{H_4}(v)=d_H(v)$ for every $v\in V(H_4)$.  There is no edge between $V(H_4)$ and $V(H)\setminus V(H_4)$,  contradicting the fact that $H$ is connected.

	Case $3$. $|E(H[\{v_1,v_2,v_3,v_4\}])|=2$. Without loss of generality, we may assume that $v_1v_2,v_3v_4\in E(H)$ and $v_5,v_6 \in N_H(v_1)$.
	
	Since $d_H(v_1)+d_H(v_2)=\Delta(H)+2$, $(v_1,v_2)$ is a full-deficiency pair. As $d_H(v_1)=d_H(v_2)=3 < \Delta(H)$, there is no common neighbor between $\{v_1,v_2\}$ and $\{v_3,v_4\}$ by Lemma \ref{L2.3}(2).
	
	Subcase $3.1$. $|N_H(v_1) \cap N_H(v_2)|=2$.
	
	Let $N_H(v_2)=\{v_1,v_5,v_6\}$. Therefore, all neighbors of $v_3$ and $v_4$ are contained in $\{v_7,\dots,\\v_{11}\}$. Let ${v_7,v_8} \in N_H(v_3)$. By the Observation, we have $v_5v_6 \notin E(H)$.
	
	If $|N_H(v_3) \cap N_H(v_4)|=0$, then let $N_H(v_4)=\{v_3,v_9,v_{10}\}$. We claim that $v_5$ is not adjacent to any vertex among $\{v_7,\dots,v_{10}\}$. Otherwise, suppose that $v_5v_7 \in E(H)$; then $H$  contains the configuration in Figure 1(b) with
	$(x,y,z,u,v,w):=(v_5,v_7,v_3,v_1,v_6,v_2)$, contradicting  {Proposition} \ref{P2.12}.  The other cases are similar.
	Thus, $v_5$ has at most three neighbors $v_1,v_2,v_{11}$, contradicting the fact  that $d_H(v_5)=4$.

	If $|N_H(v_3) \cap N_H(v_4)|=1$, then let $N_H(v_4)=\{v_3,v_7,v_9\}$. We claim that there is no edge between $\{v_5,v_6\}$ and $\{v_7,v_8,v_9\}$. Otherwise, suppose that $v_5v_7 \in E(H)$; then $H$  contains the configuration in Figure 1(b) with
	$(x,y,z,u,v,w):=(v_5,v_7,v_3,v_1,v_6,v_2)$, contradicting {Proposition} \ref{P2.12}.  The other cases are similar.
	Then $v_5$ and $v_6$ are both adjacent to $v_{10}$ and $v_{11}$. We claim that $v_{10}$ is not adjacent to any vertex among $\{v_7, v_8, v_9\}$. Otherwise, suppose that $v_7v_{10} \in E(H)$; then $H$  contains the configuration in Figure 1(e) with
	$(x,y,z,u,v,w,s):=(v_5,v_2,v_1,v_6,v_{10},v_7,v_3)$, contradicting {Proposition} \ref{P2.15}.  The other cases are similar.
	Therefore, $v_{10}$ has at most three neighbors $v_5,v_6,v_{11}$, contradicting the fact that $d_H(v_{10})=4$.

	If $|N_H(v_3) \cap N_H(v_4)| =2$, then $N_H(v_4)=\{v_3,v_7,v_8\}$ and we also have $v_7v_8 \notin E(H)$ by the Observation.
	We claim that there is no edge between $\{v_5,v_6\}$ and $\{v_7,v_8\}$. Otherwise, suppose that $v_5v_7 \in E(H)$; then $H$  contains the configuration in Figure 1(b) with
	$(x,y,z,u,v,w):=(v_5,v_7,v_3,v_1,v_6,v_2)$, contradicting {Proposition} \ref{P2.12}. The other cases are similar.
	We consider the following two subcases. If $v_5$ and $v_6$ have two common $4$-neighbors, then let $v_9,v_{10} \in N_H(v_5) \cap N_H(v_6)$. We claim that $v_7$ is not adjacent to either $v_9$ or $v_{10}$. Otherwise, suppose that $v_7v_9 \in E(H)$; then $H$  contains the configuration in Figure 1(e) with
	$(x,y,z,u,v,w,s):=(v_5,v_2,v_1,v_6,v_9,v_7,v_3)$, contradicting {Proposition} \ref{P2.15}.  The other case is similar. Hence, $v_7$ has at most three neighbors $v_3,v_4,v_{11}$, contradicting the fact  that $d_H(v_7)=4$. If $v_5$ and $v_6$ have exactly one common $4$-neighbor, then let $v_{10},v_{11} \in N_H(v_5)$ and $v_9,v_{10} \in N_H(v_6)$. We claim that $v_7$ is not adjacent to $v_{10}$. Otherwise, suppose that $v_7v_{10} \in E(H)$; then $H$  contains the configuration in Figure 1(e) with
	$(x,y,z,u,v,w,s):=(v_5,v_2,v_1,v_6,v_{10},v_7,v_3)$, contradicting {Proposition} \ref{P2.15}.
	Thus, $v_7$ is adjacent to both $v_9$ and $v_{11}$. Similarly, $v_8$ is adjacent to both $v_9$ and $v_{11}$. Since all remaining adjacencies are uniquely determined by the degree conditions, $H$ is precisely the graph shown in Figure 2(d), which has a proper $4$-edge-coloring, contradicting the fact that $H$ is a $4$-critical graph.

	Subcase $3.2$. $|N_H(v_1) \cap N_H(v_2)|=1$.
	
	Let $N_H(v_2)=\{v_1,v_5,v_{11}\}$. Therefore, all neighbors of $v_3$ and $v_4$ are contained in $\{v_7,\dots,v_{10}\}$. Let ${v_7,v_8} \in N_H(v_3)$.
	
	If $|N_H(v_3) \cap N_H(v_4)|=2$, then $N_H(v_4)=\{v_3,v_7,v_8\}$ and this reduces to Subcase 3.1.
	
	If $|N_H(v_3) \cap N_H(v_4)|=0$, then let $N_H(v_4)=\{v_3,v_9,v_{10}\}$. Note that $v_5$ is not adjacent to any vertex among $\{v_7,\dots, v_{10}\}$. Otherwise, suppose that $v_5v_7 \in E(H)$; then $H$  contains the configuration in Figure 1(d) with
	$(x,y,z,u,v):=(v_1,v_5,v_2,v_7,v_3)$, contradicting  {Proposition} \ref{P2.14}.  The other cases are similar.
	Thus, $v_5$ is adjacent to both $v_6$ and $v_{11}$. We also have that $v_6$ is not adjacent to any vertex among $\{v_7,\dots,v_{10}\}$. Otherwise, suppose that $v_6v_7 \in E(H)$; then $H$  contains the configuration in Figure 1(c) with
	$(x,y,z,u,v,w):=(v_1,v_5,v_2,v_6,v_7,v_3)$, contradicting {Proposition} \ref{P2.13}.  The other cases are similar.
	Hence, $v_6$ has at most three neighbors $v_1,v_5,v_{11}$, contradicting the fact  that $d_H(v_6)=4$.

	If $|N_H(v_3) \cap N_H(v_4)|=1$, then let $N_H(v_4)=\{v_3,v_7,v_9\}$. Note that $v_5$ is not adjacent to any vertex among $\{v_7,v_8,v_9\}$. Otherwise, suppose that $v_5v_7 \in E(H)$; then $H$ contains the configuration in Figure 1(d) with
	$(x,y,z,u,v):=(v_1,v_5,v_2,v_7,v_3)$, contradicting  {Proposition} \ref{P2.14}.  The other cases are similar.
	So $v_5$ is adjacent to exactly two vertices in $\{v_6,v_{10},v_{11}\}$. We consider the following two subcases.
	When $v_5v_6,v_5v_{11} \in E(H)$, we claim that $v_6$ is not adjacent to any vertex among $\{v_7,v_8,v_9\}$. Otherwise, suppose that $v_6v_7 \in E(H)$; then $H$  contains the configuration in Figure 1(c) with
	$(x,y,z,u,v,w):=(v_1,v_5,v_2,v_6,v_7,v_3)$, contradicting  {Proposition} \ref{P2.13}.  The other cases are similar.
	Thus, $v_6$ is adjacent to both $v_{10}$ and $v_{11}$. Similarly, we  have that $v_{11}$ is adjacent to $v_{10}$. Since all remaining adjacencies are uniquely determined by the degree conditions, $H$ is precisely the graph shown in Figure 2(e), which has a proper $4$-edge-coloring, contradicting the fact that $H$ is a $4$-critical graph.
	When $v_5v_6,v_5v_{10} \in E(H)$ (by symmetry), $v_6$ is not adjacent to any vertex among $\{v_7,v_8,v_9\}$, by the same reasoning as in the proof of the case $v_5v_6, v_5v_{11} \in E(H)$. Thus, $v_6$ is adjacent to both $v_{10}$ and $v_{11}$. Note that $v_7$ is not adjacent to $v_{11}$. Otherwise, suppose that $v_7v_{11} \in E(H)$; then $H$  contains the configuration in Figure 1(d) with
	$(x,y,z,u,v):=(v_3,v_7,v_4,v_{11},v_2)$, contradicting {Proposition} \ref{P2.14}. So $v_7$ is adjacent to exactly two vertices in $\{v_8,v_9,v_{10}\}$.
	  If $v_7v_8,v_7v_9 \in E(H)$, then  $v_{11}$ is not adjacent to either $v_8$ or $v_9$. Otherwise, suppose that $v_8v_{11} \in E(H)$; then $H$  contains the configuration in Figure 1(c) with
	$(x,y,z,u,v,w):=(v_3,v_7,v_4,v_8,v_{11},v_2)$, contradicting {Proposition} \ref{P2.13}.   The other cases are similar.
	Thus, $v_{11}$ has at most three neighbors $v_2,v_6,v_{10}$, contradicting the fact  that $d_H(v_{11})=4$.
	If $v_7v_8,v_7v_{10} \in E(H)$ (by symmetry), then it follows from the degree requirements that $v_8,v_{10},v_{11} \in N_H(v_9)$ and $v_8 \in N_H(v_{11})$. Now $H$  contains the configuration in Figure 1(c) with
	$(x,y,z,u,v,w):=(v_3,v_7,v_4,v_8,v_{11},v_2)$, contradicting {Proposition} \ref{P2.13}.

	Subcase $3.3$. $|N_H(v_1) \cap N_H(v_2)|=0$.
	
	Let $ N_H(v_2)=\{v_1,v_{10},v_{11}\}$. Therefore, all $4$-neighbors of $v_3$ and $v_4$ are contained in $\{v_7,v_8,v_9\}$ and let ${v_7,v_8} \in N(v_3)$.
	If $|N_H(v_3) \cap N_H(v_4)|=1$, then let $N_H(v_4)=\{v_3,v_7,v_9\}$, and this reduces to Subcase 3.2. If $|N_H(v_3) \cap N_H(v_4)|=2$, then let $N_H(v_4)=\{v_3,v_7,v_8\}$, and this reduces to Subcase 3.1. This completes the proof.
\end{proof}

\section{Proofs of Propositions 2.13-2.16 }

\textbf{Proposition 2.13.}
Suppose that $G$ is a graph with $\Delta(G)=4$ and contains the configuration in $Figure$ 1(c). If $\chi'(G-e)=4$, then $\chi'(G)=4$.

\begin{proof}
	Suppose to the contrary that $\chi'(G)=5$ and let $\phiv\in\CC^4(G-xy)$. Note that $xy$ is a critical edge. It follows that $(x,xy,y)$ is a multi-fan with respect to $xy$ and $\varphi$, so $\{x,y\}$ is $\varphi$-elementary by Lemma \ref{L2.2}(1). Without loss of generality, assume that $1,2 \in \overline{\varphi}(x)$ and $3 \in \overline{\varphi}(y)$. We first claim that $\phiv(xz)=3$. Otherwise, we have $\phiv(xz)=4$ and $\phiv(yz)=1$ (by symmetry between colors $1$ and $2$). Thus $2\in\overline{\varphi}(z)$ or $3\in\overline{\varphi}(z)$. Since $\phiv(yz)=1 \in \overline{\varphi}(x)$, $(y,yx,x,yz,z)$ is a multi-fan at $y$ with respect to $yx$ and $\phiv$. However, $2 \in \overline{\varphi}(x) \cap \overline{\varphi}(z)$ or $3 \in \overline{\varphi}(y) \cap \overline{\varphi}(z)$ gives a contradiction to Lemma \ref{L2.2}(1).
	Now we have $\phiv(xz)=3$. Thus $\phiv(xu)=4$. Since $(x,xy,y,xz,z)$ is a multi-fan at $x$ with respect to $xy$ and $\phiv$, $\{x,y,z\}$ is $\phiv$-elementary by Lemma \ref{L2.2}(1) and so $4 \in \overline{\varphi}(z)$ . By symmetry  between colors $1$ and $2$, we have $\phiv(yz)=1$ and $\phiv(yu)=2$. Now $\phiv(uv)=1$ or $\phiv(uv)=3$. We also claim that $\phiv(uv)=1$; otherwise, since $(x,xz,z,zy,y)$ is a $(1,3)$-chain, doing a $(1,3)$-swap along the $(1,3)$-chain containing the edge $uv$ makes $\phiv(uv)=1$, and this swap does not affect the colors of the previously considered edges.
	Note that $1,2 \in \overline{\varphi}(x)$, $3 \in \overline{\varphi}(y)$, $4 \in \overline{\varphi}(z)$, $\phiv(yz)=\phiv(uv)=1$, $\phiv(yu)=2$, $\phiv(xz)=3$, and $\phiv(xu)=4$. To reach contradictions, we consider the following three cases: $\phiv(vw)=2$, $\phiv(vw)=3$, and $\phiv(vw)=4$.
	
	Case $1$. $\phiv(vw)=2$.
	
	Assume that $1 \in \overline{\varphi}(w)$. Otherwise, we have $3 \in \overline{\varphi}(w)$ or $4 \in \overline{\varphi}(w)$. Since $(y,yx,x,yz,z)$ is a multi-fan at $y$ with respect to $yx$ and $\phiv$, $y$ and $z$ are $(3,4)$-linked by Lemma \ref{L2.2}(2). Since $(x,xy,y,xz,z)$ is a multi-fan at $x$ with respect to $xy$ and $\phiv$, $x$ and $z$ are $(1,4)$-linked by Lemma \ref{L2.2}(2). If $3 \in \overline{\varphi}(w)$, then we do a $(3,4)$-swap at $w$. Note that $x$ and $z$ remain $(1,4)$-linked after this swap. Then we do a $(1,4)$-swap at $w$, which makes $1 \in \overline{\varphi}(w)$. If $4 \in \overline{\varphi}(w)$, then doing a $(1,4)$-swap at $w$ makes $1 \in \overline{\varphi}(w)$. Since $(x,xz,z,zy,y)$ is a $(1,3)$-chain, we do a $(1,3)$-swap at $x$.  Now $(y,yu,u,uv,v,vw,w)$ is a $(1,2)$-chain, contradicting that $P_y(1,2)=P_x(1,2)$, which follows from Lemma \ref{L2.2}(2).
	
	Case $2$. $\phiv(vw)=3$.
	
	Assume that $2 \in \overline{\varphi}(w)$. Otherwise, we have $1 \in \overline{\varphi}(w)$ or $4 \in \overline{\varphi}(w)$. Since $(x,xy,y,xz,z)$ is a multi-fan at $x$ with respect to $xy$ and $\phiv$, $x$ and $z$ are both $(1,4)$-linked and
	$(2,4)$-linked by Lemma \ref{L2.2}(2). If $1 \in \overline{\varphi}(w)$, then we do a $(1,4)$-swap at $w$. Note that $x$ and $z$ remain
	$(2,4)$-linked after this swap. Then we do a $(2,4)$-swap at $w$, which makes $2 \in \overline{\varphi}(w)$. If $4 \in \overline{\varphi}(w)$, then doing a $(2,4)$-swap at $w$ makes $2 \in \overline{\varphi}(w)$. Since $(x,xy,y)$ is a multi-fan, $x$ and $y$ are $(2,3)$-linked by Lemma \ref{L2.2}(2). Doing a $(2,3)$-swap at $w$ reduces to Case 1.

	Case $3$. $\phiv(vw)=4$.
	
	Note that $1 \notin \overline{\varphi}(w)$ since otherwise, $(x,xu,u,uv,v,vw,w)$ is a $(1,4)$-chain, contradicting that $P_x(1,4)=P_z(1,4)$, which follows from Lemma \ref{L2.2}(2). Since $(y,yx,x,yz,z)$ is a multi-fan at $y$ with respect to $yx$ and $\phiv$, $y$ and $z$ are $(3,4)$-linked by Lemma \ref{L2.2}(2). Since $(x,xy,y,xz,z)$ is a multi-fan at $x$ with respect to $xy$ and $\phiv$, $x$ and $z$ are $(2,4)$-linked by Lemma \ref{L2.2}(2). If $2 \in \overline{\varphi}(w)$, then doing a $(2,4)$-swap at $w$  reduces to Case 1.
	If $3 \in \overline{\varphi}(w)$, then doing a $(3,4)$-swap at $w$  reduces to Case 2. This completes the proof.
\end{proof}

\textbf{Proposition 2.14.}
Suppose that $G$ is a graph with $\Delta(G)=4$ and contains the configuration in $Figure$ 1(d). If $\chi'(G-e)=4$, then $\chi'(G)=4$.

\begin{proof}
	Suppose to the contrary that $\chi'(G)=5$ and let  $\phiv\in\CC^4(G-xy)$. Note that $xy$ is a critical edge. It follows that $(x,xy,y)$ is a multi-fan, so $\{x,y\}$ is $\varphi$-elementary by Lemma \ref{L2.2}(1). Without loss of generality, assume that $1,2 \in \overline{\varphi}(x)$ and $3 \in \overline{\varphi}(y)$. We claim that $\phiv(xz)=3$. Otherwise, we have $\phiv(xz)=4$ and $\phiv(yz)=1$ (by symmetry between colors $1$ and $2$). Thus $2 \in\overline{\varphi}(z)$ or $3 \in\overline{\varphi}(z)$. Since $\phiv(yz)=1 \in \overline{\varphi}(x)$, $(y,yx,x,yz,z)$ is a multi-fan at $y$ with respect to $yx$ and $\phiv$. However, $2 \in \overline{\varphi}(x) \cap \overline{\varphi}(z)$ or $3 \in \overline{\varphi}(y) \cap \overline{\varphi}(z)$ gives a contradiction to Lemma \ref{L2.2}(1). Now we have $\phiv(xz)=3$. Since $(x,xy,y,xz,z)$ is a multi-fan at $x$ with respect to $xy$ and $\phiv$, $\{x,y,z\}$ is $\phiv$-elementary by Lemma \ref{L2.2}(1) and so $4 \in \overline{\varphi}(z)$.  By symmetry between colors $1$ and $2$, we have $\phiv(yz)=1$.
	Note that $1,2 \in \overline{\varphi}(x)$, $3 \in \overline{\varphi}(y)$, $4 \in \overline{\varphi}(z)$, $\phiv(yz)=1$, and $\phiv(xz)=3$. To reach contradictions, we consider the following two cases: $\phiv(yu)=2$ and $\phiv(yu)=4$.
	
	Case $1$. $\phiv(yu)=2$.
	
	Subcase $1.1$. $\phiv(uv)=1$.

	Assume that $2 \in \overline{\varphi}(v)$. Otherwise, we have $3\in \overline{\varphi}(v)$ or $4 \in \overline{\varphi}(v)$. By Lemma \ref{L2.2}(2), since $(y,yx,x,yz,z)$ is a multi-fan at $y$ with respect to $yx$ and $\phiv$, it follows that $y$ and $x$ are $(2,3)$-linked, and $y$ and $z$ are
	$(3,4)$-linked. If $3 \in \overline{\varphi}(v)$, then doing a $(2,3)$-swap at $v$ makes $2 \in \overline{\varphi}(v)$. If $4 \in \overline{\varphi}(v)$, then we do a $(3,4)$-swap at $v$. Note that $x$ and $y$ remain $(2,3)$-linked after this swap. Then doing a $(2,3)$-swap at $v$ makes $2 \in \overline{\varphi}(v)$. Since $(x,xz,z,zy,y)$ is a $(1,3)$-chain, we do a $(1,3)$-swap along the $(1,3)$-chain containing the edge $uv$.  Now $(y,yu,u,uv,v)$ is a $(2,3)$-chain, contradicting that $P_y(2,3)=P_x(2,3)$, which follows from Lemma \ref{L2.2}(2).

	Subcase $1.2$. $\phiv(uv)=3$.
	
	Since $(x,xz,z,zy,y)$ is a $(1,3)$-chain, we do a $(1,3)$-swap along the $(1,3)$-chain containing the edge $uv$. This reduces to Subcase 1.1.
	
	Subcase $1.3$. $\phiv(uv)=4$.

	Assume that $1 \in \overline{\varphi}(v)$. Otherwise, we have $2 \in \overline{\varphi}(v)$ or $3 \in \overline{\varphi}(v)$. Since $(x,xy,y)$ is a multi-fan, $x$ and $y$ are both $(1,3)$-linked and  $(2,3)$-linked by Lemma \ref{L2.2}(2). If $2 \in \overline{\varphi}(v)$, then we do a $(2,3)$-swap at $v$. Note that $x$ and $y$ remain
	$(1,3)$-linked after this swap. Then doing a $(1,3)$-swap at $v$ makes $1 \in \overline{\varphi}(v)$. If $3 \in \overline{\varphi}(v)$, then doing a $(1,3)$-swap at $v$ makes $1 \in \overline{\varphi}(v)$. Since $(x,xy,y,xz,z)$ is a multi-fan at $x$ with respect to $xy$ and $\phiv$, $x$ and $z$ are $(1,4)$-linked by Lemma \ref{L2.2}(2). We do a $(1,4)$-swap at $v$, which reduces to Subcase 1.1.

	Case $2$. $\phiv(yu)=4$.
	
	Subcase $2.1$. $\phiv(uv)=1$.
	
	Note that $4 \notin \overline{\varphi}(v)$ since otherwise, $(z,zy,y,yu,u,uv,v)$ is a $(1,4)$-chain, contradicting that $P_z(1,4)=P_x(1,4)$, which follows from Lemma \ref{L2.2}(2).
	Assume that $2 \in \overline{\varphi}(v)$. If not, then $3 \in \overline{\varphi}(v)$. Since $x$ and $y$ are $(2,3)$-linked, we do a $(2,3)$-swap at $v$, which makes $2 \in \overline{\varphi}(v)$. Now we have $2 \in \overline{\varphi}(v)$. Since $(x,xy,y,xz,z)$ is a multi-fan at $x$ with respect to $xy$ and $\phiv$, $x$ and $z$ are $(2,4)$-linked by Lemma \ref{L2.2}(2). We do a $(2,4)$-swap at $v$. If $yu \notin P_v(2,4)$, then $(z,zy,y,yu,u,uv,v)$ is a $(1,4)$-chain, contradicting that $P_z(1,4)=P_x(1,4)$, which follows from Lemma \ref{L2.2}(2). If $yu \in P_v(2,4)$, then the case reduces to Case 1.

	Subcase $2.2$. $\phiv(uv)=2$.

	Assume that $4 \in \overline{\varphi}(v)$. Otherwise, we have $1 \in \overline{\varphi}(v)$ or $3 \in \overline{\varphi}(v)$.  Since $(y,yx,x,yz,z)$ is a multi-fan at $y$ with respect to $yx$ and $\phiv$, $y$ and $z$ are $(3,4)$-linked by Lemma \ref{L2.2}(2). Since $(x,xy,y,xz,z)$ is a multi-fan at $x$ with respect to $xy$ and $\phiv$, $x$ and $z$ are $(1,4)$-linked by Lemma \ref{L2.2}(2). If $1 \in \overline{\varphi}(v)$, then doing a $(1,4)$-swap at $v$ makes $4 \in \overline{\varphi}(v)$. If $3 \in \overline{\varphi}(v)$, then doing a $(3,4)$-swap at $v$ makes $4 \in \overline{\varphi}(v)$. Since $(x,xy,y,xz,z)$ is a multi-fan at $x$ with respect to $xy$ and $\phiv$, $x$ and $z$ are $(2,4)$-linked by Lemma \ref{L2.2}(2). We do a $(2,4)$-swap at $v$, which reduces to Case 1.
	
	Subcase $2.3$. $\phiv(uv)=3$.
	
	Since $(x,xz,z,zy,y)$ is a $(1,3)$-chain, we do a $(1,3)$-swap along the  $(1,3)$-chain containing the edge $uv$. Then the case reduces to Subcase 2.1. This completes the proof.
\end{proof}

\textbf{Proposition 2.15.}
Suppose that $G$ is a graph with $\Delta(G)=4$ and contains the configuration in $Figure$ 1(e). If $\chi'(G-e)=4$, then $\chi'(G)=4$.

\begin{proof}
	Suppose to the contrary that $\chi'(G)=5$ and let $\phiv\in\CC^4(G-xy)$. Note that $xy$ is a critical edge. It follows that $(x,xy,y)$ is a multi-fan, so $\{x,y\}$ is $\varphi$-elementary by Lemma \ref{L2.2}(1).  Without loss of generality, assume that $1,2 \in \overline{\varphi}(y)$ and $3 \in \overline{\varphi}(x)$. We first claim that $\phiv(zy)=3$. Otherwise, we have $\phiv(zy)=4$ and $\phiv(zx)=1$ (by symmetry between colors $1$ and $2$). Thus $2 \in\overline{\varphi}(z)$ or $3 \in\overline{\varphi}(z)$. Since $\phiv(zx)=1 \in \overline{\varphi}(y)$, $(x,xy,y,xz,z)$ is a multi-fan at $x$ with respect to $xy$ and $\phiv$. However, $2 \in \overline{\varphi}(z) \cap \overline{\varphi}(y)$ or $3 \in \overline{\varphi}(z) \cap \overline{\varphi}(x)$ gives a contradiction to Lemma \ref{L2.2}(1). Now we have $\phiv(zy)=3$. Thus $\phiv(uy)=4$. Since $(y,yx,x,yz,z)$ is a multi-fan at $y$ with respect to $yx$ and $\phiv$, $\{y,x,z\}$ is $\phiv$-elementary by Lemma \ref{L2.2}(1) and so $4 \in \overline{\varphi}(z)$. By symmetry  between colors $1$ and $2$, we have $\phiv(zx)=1$ and $\phiv(zu)=2$. Note that $(z,zu,u,uy,y)$ is a $(2,4)$-chain and $(x,xz,z,zy,y)$ is a $(1,3)$-chain. We also claim that $\phiv(xv)=2$. Otherwise, we have $\phiv(xv)=4$; then doing a $(2,4)$-swap along the $(2,4)$-chain containing the edge $xv$ makes $\phiv(xv)=2$. Finally, we claim that $\phiv(uv)=1$. Otherwise, we have $\phiv(uv)=3$; then doing a $(1,3)$-swap along the $(1,3)$-chain containing the edge $uv$ makes $\phiv(uv)=1$.
	Note that $1,2 \in \overline{\varphi}(y)$, $3 \in \overline{\varphi}(x)$, $4 \in \overline{\varphi}(z)$, $\phiv(zx)=\phiv(uv)=1$, $\phiv(zu)=\phiv(xv)=2$, $\phiv(zy)=3$, and $\phiv(uy)=4$.
	To reach contradictions, we consider the following two cases: $\phiv(vw)=3$ and $\phiv(vw)=4$.

	Case $1$. $\phiv(vw)=3$.
	
	Subcase $1.1$. $\phiv(ws)=1$.
	
	Assume that $2 \in \overline{\varphi}(s)$. Otherwise, we have $3 \in \overline{\varphi}(s)$ or $4 \in \overline{\varphi}(s)$.
	When $3 \in \overline{\varphi}(s)$,
	we do a $(2,3)$-swap at $s$ since $x$ and $y$ are $(2,3)$-linked, which reduces to $2 \in \overline{\varphi}(s)$.
	When $4 \in \overline{\varphi}(s)$, we have $xv\in P_s(2,4)$; otherwise, since $z$ and $y$ are $(2,4)$-linked, we do a $(2,4)$-swap at $s$, which reduces to $2 \in \overline{\varphi}(s)$. If $P_s(2,4)$ meets $v$ before $x$, then we do a $(2,4)$-swap on $P_{[s,v]}(2,4)$, color $xy$ with $2$, and uncolor $xv$. Since $(x,xz,z,zy,y)$ is a $(1,3)$-chain, we do a $(1,3)$-swap at $x$. Now $(v,vu,u,uy,y,yz,z)$ is a $(1,4)$-chain, contradicting that $P_v(1,4)=P_x(1,4)$, which follows from Lemma \ref{L2.2}(2). If $P_s(2,4)$ meets $x$ before $v$, then we do a $(2,4)$-swap on $P_{[s,x]}(2,4)$, recolor $zx:1\rightarrow4$, $zy:3\rightarrow1$, color $yx$ with $3$, and uncolor $xv$. Now $(v,vu,u,uz,z,zy,y)$ is a $(1,2)$-chain, contradicting that $P_v(1,2)=P_x(1,2)$, which follows from Lemma \ref{L2.2}(2).
	
	Now $2 \in \overline{\varphi}(s)$,
	we do a $(1,2)$-swap at $s$ since $(z,zu,u,uv,v,vx,x,xz,z)$ is a $(1,2)$-chain. Denote the new edge-coloring by $\phiv_1$. Note that $vu,vw \in P_s(1,3,\phiv_1)$;  otherwise, since $(x,xz,z,zy,y)$ is a $(1,3)$-chain, a $(1,3)$-swap at $s$ would make $(x,xv,v,vw,w,ws,s)$ a $(2,3)$-chain, contradicting that $P_x(2,3) = P_y(2,3)$, which follows from Lemma \ref{L2.2}(2).
	If $P_s(1,3,\phiv_1)$ meets $w$ before $u$, then we do a $(1,3)$-swap on $P_{[s,v]}(1,3,\phiv_1)$, recolor $xv:2\rightarrow3$, color $xy$ with $2$ and uncolor $uv$. Now $(u,uz,z,zx,x,xy,y)$ is a $(1,2)$-chain, contradicting that $P_u(1,2)=P_v(1,2)$, which follows from Lemma \ref{L2.2}(2). If $P_s(1,3,\phiv_1)$ meets $u$ before $w$, then we do a $(1,3)$-swap on $P_{[s,u]}(1,3,\phiv_1)$, recolor $zy:3\rightarrow2$, $zu:2\rightarrow3$, color $xy$ with $3$, and uncolor $uv$. Now $(v,vx,x,xz,z,zy,y)$ is a $(1,2)$-chain, contradicting that $P_v(1,2)=P_u(1,2)$, which follows from Lemma \ref{L2.2}(2).

	Subcase $1.2$. $\phiv(ws)=2$.
	
	Note that $(z,zu,u,uv,v,vx,x,xz,z)$ is a $(1,2)$-chain. We do a $(1,2)$-swap along the  $(1,2)$-chain containing the edge $ws$, which reduces to Subcase 1.1.
	
	Subcase $1.3$. $\phiv(ws)=4$.

	Assume that $1 \in \overline{\varphi}(s)$. Otherwise, we have $2 \in \overline{\varphi}(s)$ or $3 \in \overline{\varphi}(s)$. Note that $(z,zu,u,uv,v,\allowbreak vx,x,xz,z)$ is a $(1,2)$-chain. If $2 \in \overline{\varphi}(s)$, then doing a $(1,2)$-swap at $s$ makes $1 \in \overline{\varphi}(s)$. If $3 \in \overline{\varphi}(s)$, then since $x$ and $y$ are $(2,3)$-linked, doing a $(2,3)$-swap at $s$ makes $2 \in \overline{\varphi}(s)$. Now we have $1 \in \overline{\varphi}(s)$. Since $(y,yx,x,yz,z)$ is a multi-fan at $y$ with respect to $yx$ and $\phiv$, $z$ and $y$ are $(1,4)$-linked by Lemma \ref{L2.2}(2). We do a $(1,4)$-swap at $s$, which reduces to Subcase 1.1.
	
	Case $2$. $\phiv(vw)=4$.
	
	Subcase $2.1$. $\phiv(ws)=1$.
	
	Note that $4 \notin \overline{\varphi}(s)$ since otherwise, $(y,yu,u,uv,v,vw,w,ws,s)$ is a $(1,4)$-chain, contradicting that $P_y(1,4)=P_z(1,4)$, which follows from Lemma \ref{L2.2}(2).
	  By Lemma \ref{L2.2}(2), since $(y,yx,x,yz,z)$ is a multi-fan at $y$ with respect to $yx$ and $\phiv$, $x$ and $y$ are $(2,3)$-linked, and $z$ and $y$ are $(2,4)$-linked. Assume that $2 \in \overline{\varphi}(s)$; if not, then $3 \in \overline{\varphi}(s)$, and doing a $(2,3)$-swap at $s$ makes $2 \in \overline{\varphi}(s)$. Note that $vx,vw \in P_s(2,4)$; otherwise, doing a $(2,4)$-swap at $s$ makes $(y,yu,u,uv,v,vw,w,ws,s)$ a $(1,4)$-chain, contradicting that $P_y(1,4)=P_z(1,4)$, which follows from Lemma \ref{L2.2}(2). If $P_s(2,4)$ meets $w$ before $x$, then we do a $(2,4)$-swap on $P_{[s,v]}(2,4)$, color $xy$ with $2$, and uncolor $xv$. Since $(x,xz,z,zy,y)$ is a $(1,3)$-chain, we do a $(1,3)$-swap at $x$ and thus $(v,vu,u,uy,y,yz,z)$ is a $(1,4)$-chain, contradicting that $P_v(1,4)=P_x(1,4)$, which follows from Lemma \ref{L2.2}(2). If $P_s(2,4)$ meets $x$ before $w$, then we do a $(2,4)$-swap on $P_{[s,x]}(2,4)$, recolor $zx:1\rightarrow4$, $zy:3\rightarrow1$,  color $xy$ with $3$, and uncolor $xv$. Now
	$(v,vu,u,uz,z,zy,y)$ is a $(1,2)$-chain, contradicting that $P_v(1,2)=P_x(1,2)$, which follows from Lemma \ref{L2.2}(2).
	
	Subcase $2.2$. $\phiv(ws)=2$.
	
	Since $(z,zu,u,uv,v,vx,x,xz,z)$ is a $(1,2)$-chain, we do a $(1,2)$-swap along the  $(1,2)$-chain containing the edge $ws$, which reduces to Subcase 2.1.
	
	Subcase $2.3$. $\phiv(ws)=3$.
	
	By Lemma \ref{L2.2}(2), since $(x,xy,y,xz,z)$ is a multi-fan at $x$ with respect to $xy$ and $\phiv$, it follows that $x$ and $y$ are $(2,3)$-linked, and $x$ and $z$ are $(3,4)$-linked. If $1 \in \overline{\varphi}(s)$, then doing $(1,2)-(2,3)$-swaps at $s$ reduces to Subcase 2.2. If $2 \in \overline{\varphi}(s)$, then doing a $(2,3)$-swap at $s$ reduces to Subcase 2.2. If $4 \in \overline{\varphi}(s)$, then doing a $(3,4)$-swap at $s$ reduces to Case 1.	This completes the proof.
\end{proof}

\textbf{Proposition 2.16.}
Suppose that $G$ is a graph with $\Delta(G)=4$ and contains the configuration in $Figure$ 1(f). If $\chi'(G-e)=4$, then $\chi'(G)=4$.

\begin{proof}
	
	Suppose to the contrary that $\chi'(G)=5$ and let  $\phiv\in\CC^4(G-xy)$. Note that $xy$ is a critical edge. It follows that $(x,xy,y)$ is a multi-fan, so $\{x,y\}$ is $\varphi$-elementary by Lemma \ref{L2.2}(1). Without loss of generality, assume that $1,2 \in \overline{\varphi}(x)$ and $3 \in \overline{\varphi}(y)$. To reach contradictions, we consider the following two cases: $\phiv(yz)=1$ (by symmetry between colors $1$ and $2$) and $\phiv(yz)=4$.
	
	Case $1$. $\phiv(yz)=1$.	
	
	Since $\phiv(yz)=1 \in \overline{\varphi}(x)$, $(y,yx,x,yz,z)$ is a multi-fan at $y$ with respect to $yx$ and $\phiv$. By Lemma \ref{L2.2}(1),  $\{y,x,z\}$ is $\varphi$-elementary and thus $4 \in \overline{\varphi}(z)$. We consider the following two subcases: $\phiv(yv)=2$ and $\phiv(yv)=4$.
	
	Subcase $1.1$.  $\phiv(yv)=2$.
	
	%By symmetry between $vu$ and $vw$, we consider the following three subcases.
	
	Subcase $1.1.1$. $\phiv(vu)=1$ and $\phiv(vw)=3$.
	
	Since $(x,xy,y)$ is a multi-fan, $x$ and $y$ are both $(2,3)$-linked and $(1,3)$-linked by Lemma \ref{L2.2}(2). Note that $2 \notin \overline{\varphi}(w)$ since otherwise, $(y,yv,v,vw,w)$ is a $(2,3)$-chain, contradicting that $P_y(2,3)=P_x(2,3)$, which follows from Lemma \ref{L2.2}(2).  Now $1 \in \overline{\varphi}(w)$ or $4 \in \overline{\varphi}(w)$.

	When $1 \in \overline{\varphi}(w)$, we claim that $4 \in \overline{\varphi}(u)$. Otherwise, we have $2 \in \overline{\varphi}(u)$ or $3 \in \overline{\varphi}(u)$. If $2 \in \overline{\varphi}(u)$, then we do a $(1,3)$-swap at $x$. Note that $(y,yv,v,vu,u)$ is a $(1,2)$-chain, contradicting that  $P_y(1,2)=P_x(1,2)$, which follows from Lemma \ref{L2.2}(2). If $3 \in \overline{\varphi}(u)$, then doing a $(2,3)$-swap at $u$ reduces to $2 \in \overline{\varphi}(u)$. Now we have $4 \in \overline{\varphi}(u)$. Note that
	$1 \in \overline{\varphi}(w)$, $\phiv(vw)=3$, and 	$\phiv(vu)=1$.
	Since $x$ and $y$ are $(1,3)$-linked, we do a $(1,3)$-swap at $w$. Denote the new edge-coloring by $\phiv_1$. Since $\phiv_1(yz)=1 \in \overline{\varphi_1}(x)$, $(y,yx,x,yz,z)$ is a multi-fan at $y$ with respect to $yx$ and $\phiv_1$. By Lemma \ref{L2.2}(3),
	$x$ and $z$ are $(2,4)$-linked. We claim that $yv \in P_u(2,4,\phiv_1)$; otherwise, we do a $(2,4)$-swap at $u$, and then $(y,yv,v,vu,u)$ is a $(2,3)$-chain,  contradicting that $P_y(2,3)=P_x(2,3)$, which follows from Lemma \ref{L2.2}(2). When $P_u(2,4,\phiv_1)$ meets $v$ before $y$, we do a $(2,4)$-swap on $P_{[u,v]}(2,4,\phiv_1)$, recolor $yv:2\rightarrow3$, color $yx$ with $2$ and uncolor $vu$. Denote the new edge-coloring by $\phiv_{11}$. Note that $2,3 \in \overline{\varphi_{11}}(u)$ and $4\in \overline{\varphi_{11}}(v)$.
	Since $(v,vu,u)$ is a multi-fan, $v$ and $u$ are $(3,4)$-linked by Lemma \ref{L2.2}(2). Then we do a $(3,4)$-swap at $v$ and recolor $vw:1\rightarrow3$. Now $(v,vy,y,yz,z)$ is a $(1,4)$-chain, contradicting that $P_v(1,4)=P_u(1,4)$, which follows from Lemma \ref{L2.2}(2). When $P_u(2,4,\phiv_1)$ meets $y$ before $v$, we do a $(2,4)$-swap on $P_{[u,y]}(2,4,\phiv_1)$, recolor $vu:3\rightarrow2$, $yv:2\rightarrow3$, $yz:1\rightarrow4$, and color $yx$ with $1$. Now we get a new (proper) $4$-edge-coloring of $G$, contradicting that  $\chi'(G)=5$.
	
	When $4 \in \overline{\varphi}(w)$, we claim that $4 \in \overline{\varphi}(u)$. Otherwise, we have $2 \in \overline{\varphi}(u)$ or $3 \in \overline{\varphi}(u)$. If $2 \in \overline{\varphi}(u)$, then doing $(2,3)-(1,3)$-swaps at $u$ reduces to $1 \in \overline{\varphi}(w)$ and $4 \in \overline{\varphi}(u)$. If $3 \in \overline{\varphi}(u)$, then doing a $(1,3)$-swap at $u$ reduces to $1 \in \overline{\varphi}(w)$ and $4 \in \overline{\varphi}(u)$. Now we have $4 \in \overline{\varphi}(u)$. Next we assume that $vu,vw \notin P_y(1,3)=P_x(1,3)$. Otherwise, we have $vu,vw \in P_y(1,3)=P_x(1,3)$. We discuss the following two subcases to get contradictions.
	If $P_y(1,3)$ meets $w$ before $u$, then we do a $(1,3)$-swap on $P_{[y,w]}(1,3)$, color $yx$ with $1$, and uncolor $vw$. Denote the new edge-coloring by $\phiv_1$. Since $\phiv_1(vu)=1 \in \overline{\varphi_1}(w)$, $(v,vw,w,vu,u)$ is a multi-fan at $v$ with respect to $vw$ and $\phiv_1$. However, $4 \in \overline{\varphi_1}(u) \cap \overline{\varphi_1}(w)$ gives a contradiction to Lemma \ref{L2.2}(1). A similar argument works if $P_y(1,3)$ meets $u$ before $w$. Thus $vu,vw \notin P_y(1,3)=P_x(1,3)$. Since $(y,yx,x,yz,z)$ is a multi-fan at $y$ with respect to $yx$ and $\phiv$, $x$ and $z$ are $(2,4)$-linked by Lemma \ref{L2.2}(3). Note that $yv \in P_w(2,4)$;
	otherwise, doing a $(2,4)$-swap at $w$ makes $(y,yv,v,vw,w)$ a $(2,3)$-chain, contradicting that $P_y(2,3)=P_x(2,3)$, which follows from Lemma \ref{L2.2}(2).
	If $P_w(2,4)$ meets $y$ before $v$, then we do a $(2,4)$-swap on $P_{[w,y]}(2,4)$, recolor $yz:1\rightarrow4$, $vw:3\rightarrow2$, $yv:2\rightarrow3$, and color $yx$ with $1$. Now we get a new (proper) $4$-edge-coloring of $G$, contradicting that  $\chi'(G)=5$. If $P_w(2,4)$ meets $v$ before $y$, then we do a $(1,3)$-swap along the $(1,3)$-chain containing the edges $vu$ and $vw$. Then we do a $(2,4)$-swap on $P_{[w,v]}(2,4)$, recolor $vu:3\rightarrow4$,  $yv:2\rightarrow3$, and color $yx$ with $2$. Now we get a new (proper) $4$-edge-coloring of $G$, contradicting that  $\chi'(G)=5$.

	Subcase $1.1.2$. $\phiv(vu)=1$ and $\phiv(vw)=4$.
	
	By Lemma \ref{L2.2}(2), since $(y,yx,x,yz,z)$ is a multi-fan at $y$ with respect to $yx$ and $\phiv$, it follows that $y$ and $x$ are $(2,3)$-linked, and $y$ and $z$ are $(3,4)$-linked.
	Assume that $1 \in \overline{\varphi}(w)$. Otherwise, if $3 \in \overline{\varphi}(w)$, then doing a $(3,4)$-swap at $w$ reduces to Subcase 1.1.1; if $2 \in \overline{\varphi}(w)$, then doing a $(2,3)$-swap at $w$ reduces to $3 \in \overline{\varphi}(w)$.
	Now we have $1 \in \overline{\varphi}(w)$.
	
	Next we claim that $4 \in \overline{\varphi}(u)$.  Otherwise, if $3 \in \overline{\varphi}(u)$, then doing a $(3,4)$-swap at $u$ reduces to $4 \in \overline{\varphi}(u)$ (if $vw \notin P_u(3,4)$) or Subcase 1.1.1 (if $vw \in P_u(3,4)$); if $2 \in \overline{\varphi}(u)$, then doing a $(2,3)$-swap at $u$ reduces to $3 \in \overline{\varphi}(u)$. Now we have $4 \in \overline{\varphi}(u)$. Note that $x$ and $y$ are $(1,3)$-linked. When $vu \notin P_w(1,3)$,  doing $(1,3)-(3,4)$-swaps at $w$ reduces to Subcase 1.1.1. When $vu \in P_w(1,3)$, we do a $(1,3)$-swap at $w$. Denote the new edge-coloring by $\phiv_1$. Since $ \phiv_1(yz)=1 \in \overline{\varphi_1}(x)$, $(y,yx,x,yz,z)$ is a multi-fan at $y$ with respect to $yx$ and $\phiv_1$. By Lemma \ref{L2.2}(3),  $x$ and $z$ are $(2,4)$-linked. Note that $vy,vw \in P_u(2,4,\phiv_1)$; otherwise, we do a $(2,4)$-swap at $u$, and thus $(y,yv,v,vu,u)$ is a $(2,3)$-chain, contradicting that $P_y(2,3)=P_x(2,3)$, which follows from Lemma \ref{L2.2}(2). If $P_u(2,4,\phiv_1)$ meets $y$ before $w$, then we do a $(2,4)$-swap on $P_{[u,y]}(2,4,\phiv_1)$, recolor $yz:1\rightarrow4$, $vu:3\rightarrow2$, $yv:2\rightarrow3$, and color $yx$ with $1$.  Now we get a new (proper) $4$-edge-coloring of $G$, contradicting that $\chi'(G)=5$. If $P_u(2,4,\phiv_1)$ meets $w$ before $y$, then we do a $(2,4)$-swap on $P_{[u,v]}(2,4,\phiv_1)$, recolor $yv:2\rightarrow3$, color $yx$ with $2$ and uncolor $vu$. Denote the new edge-coloring by $\phiv_2$. Note that $vu$ is a critical edge. Since $\phiv_2(vw)=2 \in \overline{\varphi_2}(u)$, $(v,vu,u,vw,w)$ is a multi-fan at $v$ with respect to $vu$ and $\phiv_2$. However, $3 \in \overline{\varphi_2}(u) \cap \overline{\varphi_2}(w)$ gives a contradiction to Lemma \ref{L2.2}(1).

	Subcase $1.1.3$. $\phiv(vu)=3$ and $\phiv(vw)=4$.
	
	Since $(x,xy,y)$ is a multi-fan, $x$ and $y$ are both $(1,3)$-linked and  $(2,3)$-linked by Lemma \ref{L2.2}(2). First we assume that $2 \in \overline{\varphi}(w)$. Otherwise, if $3 \in \overline{\varphi}(w)$, then doing a $(2,3)$-swap at $w$ reduces to $2 \in \overline{\varphi}(w)$; if $1 \in \overline{\varphi}(w)$, then doing a $(1,3)$-swap at $w$ reduces to Subcase 1.1.2 (if $vu \in P_w(1,3)$) or $3 \in \overline{\varphi}(w)$ (if $vu \notin P_w(1,3)$). Now we have $2 \in \overline{\varphi}(w)$.
	
	Note that $2\notin \overline{\varphi}(u)$ since otherwise, $(y,yv,v,vu,u)$ is a $(2,3)$-chain, contradicting that $P_y(2,3)=P_x(2,3)$, which follows from Lemma \ref{L2.2}(2).
	If $1 \in \overline{\varphi}(u)$, then doing a $(1,3)$-swap at $u$ reduces to Subcase 1.1.2. So we only consider $4 \in \overline{\varphi}(u)$. When $P_w(2,4) \neq P_u(2,4)$, we do a $(2,4)$-swap at $w$, and thus $(y,yv,v,vu,u)$ is a $(3,4)$-chain, contradicting that $P_y(3,4)=P_z(3,4)$, which follows from Lemma \ref{L2.2}(2).
	When $P_w(2,4) = P_u(2,4)$, doing a $(2,4)$-swap at $w$ and a $(2,3)$-swap at $u$ makes $(y,yv,v,vw,w)$  a $(3,4)$-chain, contradicting that $P_y(3,4)=P_z(3,4)$, which follows from Lemma \ref{L2.2}(2).

	Subcase $1.2$.  $\phiv(yv)=4$.
	
	Subcase $1.2.1$. $\phiv(vu)=1$ and $\phiv(vw)=2$.
	
	All cases can reduce to Subcase 1.1. If $4 \in \overline{\varphi}(w)$, then doing a $(2,4)$-swap at $w$ reduces to Subcase 1.1. If $3 \in \overline{\varphi}(w)$, then doing $(3,4)-(2,4)$-swaps at $w$ reduces to Subcase 1.1.  If $1 \in \overline{\varphi}(w)$, then we do a $(1,3)$-swap at $w$ since $x$ and $y$ are $(1,3)$-linked. Then doing $(3,4)-(2,4)$-swaps at $w$ reduces to Subcase 1.1.

	Subcase $1.2.2$. $\phiv(vu)=1$ and $\phiv(vw)=3$.
	
	We first claim that $2 \in \overline{\varphi}(u)$. Otherwise, if $3 \in \overline{\varphi}(u)$, then doing a $(2,3)$-swap at $u$ reduces to Subcase 1.2.1 (if $vw \in P_u(2,3)$) or $2 \in \overline{\varphi}(u)$ (if $vw \notin P_u(2,3)$); if $4 \in \overline{\varphi}(u)$, then doing a $(3,4)$-swap at $u$ reduces to $3 \in \overline{\varphi}(u)$. Now we have $2 \in \overline{\varphi}(u)$.
	
	Note that $4 \notin \overline{\varphi}(w)$ since $y$ and $z$ are $(3,4)$-linked. We discuss the following two cases.  If $2 \in \overline{\varphi}(w)$, then doing a $(2,3)$-swap at $w$ reduces to Subcase 1.2.1. If $1 \in \overline{\varphi}(w)$, then we do a $(1,3)$-swap at $w$. Denote the new edge-coloring by $\phiv_1$. Since $(y,yx,x,yz,z)$ is a multi-fan at $y$ with respect to $yx$ and $\phiv_1$, $x$ and $z$ are $(2,4)$-linked by Lemma \ref{L2.2}(3).
	When $yv \notin P_x(2,4,\phiv_1)=P_z(2,4,\phiv_1)$,  doing a $(2,4)$-swap along the $(2,4)$-chain containing the edge $yv$ reduces to Subcase 1.1. When $yv \in P_x(2,4,\phiv_1)=P_z(2,4,\phiv_1)$, we do a $(2,4)$-swap at $u$, and thus $(y,yv,v,vu,u)$ is a $(3,4)$-chain, contradicting that $P_y(3,4)=P_z(3,4)$, which follows from Lemma \ref{L2.2}(2).

	Subcase $1.2.3$. $\phiv(vu)=2$ and $\phiv(vw)=3$.
	
	If $4 \in \overline{\varphi}(u)$, then doing a $(2,4)$-swap at $u$ reduces to Subcase 1.1. If $3 \in \overline{\varphi}(u)$, then doing $(3,4)-(2,4)$-swaps at $u$ reduces to Subcase 1.1. If $1 \in \overline{\varphi}(u)$, then doing a $(1,3)$-swap at $u$ reduces to Subcase 1.2.1 (if $vw \in P_u(1,3)$) or $3 \in \overline{\varphi}(u)$ (if $vw \notin P_u(1,3)$).
	
	Case $2$. $\phiv(yz)=4$.	
	
	Assume that $3 \in \overline{\varphi}(z)$. Otherwise, we have $1 \in \overline{\varphi}(z)$ (by symmetry  between colors $1$ and $2$). Since $x$ and $y$ are $(1,3)$-linked, doing a $(1,3)$-swap at $z$ makes $3 \in \overline{\varphi}(z)$. Thus $3 \in \overline{\varphi}(z)$ and $\phiv(yv) =1$ (by symmetry between colors $1$ and $2$). By symmetry between $vu$ and $vw$, we consider the following three subcases.
	
	Subcase $2.1$. $\phiv(vu)=2$ and $\phiv(vw)=3$.
	
	Note that $1 \notin \overline{\varphi}(w)$ since otherwise, $(y,yv,v,vw,w)$ is a $(1,3)$-chain, contradicting that $P_y(1,3)=P_x(1,3)$, which follows from Lemma \ref{L2.2}(2).
	Now $2\in \overline{\varphi}(w)$ or $4 \in \overline{\varphi}(w)$. Assume that $2 \in \overline{\varphi}(w)$. Otherwise, since $y$ and $z$ are $(3,4)$-linked, we do a $(3,4)$-swap at $w$. Denote the new edge-coloring by $\phiv_1$. If $P_z(1,3,\phiv_1)=P_w(1,3,\phiv_1)$, then we do a $(1,3)$-swap at $z$. If $P_z(1,3,\phiv_1)\neq P_w(1,3,\phiv_1)$, then we do a $(1,3)$-swap at both $z$ and $w$. Denote the new edge-coloring by $\phiv_2$. Now $(z,zy,y,yv,v,vw,w)$ is a $(1,4)$-chain, we do a $(1,4)$-swap at $z$ which reduces to Case 1. Now we have $2 \in \overline{\varphi}(w)$. We discuss the following three cases.
	
	If $1 \in \overline{\varphi}(u)$, then we do a $(2,3)$-swap at $x$ and thus $(y,yv,v,vu,u)$ is a $(1,2)$-chain, contradicting that $P_y(1,2)=P_x(1,2)$, which follows from Lemma \ref{L2.2}(2).
	If $3 \in \overline{\varphi}(u)$, then we do a $(1,3)$-swap at $u$, followed by a $(2,3)$-swap at $w$. Note that $(y,yv,v,vu,u)$ is a $(1,3)$-chain, contradicting that $P_y(1,3)=P_x(1,3)$, which follows from Lemma \ref{L2.2}(2).
	If $4 \in \overline{\varphi}(u)$, then we do a $(1,3)$-swap at $z$, followed by a $(2,3)$-swap at $w$. Denote the new edge-coloring by $\phiv_1$. Note that $1\in\overline{\varphi_1}(x)\cap\overline{\varphi_1}(z)$.
	When $P_x(1,4,\phiv_1) \neq P_z(1,4,\phiv_1)$,  doing a $(1,4)$-swap at $z$ reduces to Subcase 1.2. When $P_x(1,4,\phiv_1) = P_z(1,4,\phiv_1)$, we do a $(1,4)$-swap at $u$. Note that $(y,yv,v,vu,u)$ is a $(1,3)$-chain, contradicting that $P_y(1,3)=P_x(1,3)$, which follows from Lemma \ref{L2.2}(2).
	
	Subcase $2.2$. $\phiv(vu)=2$ and $\phiv(vw)=4$.
	
	Since $(y,yz,z)$ is a $(3,4)$-chain, we do a $(3,4)$-swap along the $(3,4)$-chain containing the edge $vw$, which reduces to Subcase 2.1.
	
	Subcase $2.3$. $\phiv(vu)=3$ and $\phiv(vw)=4$.
	
	If $1 \in \overline{\varphi}(w)$, then we recolor $yz:4\rightarrow3$ and thus $(y,yv,v,vw,w)$ is a $(1,4)$-chain, contradicting that $P_y(1,4)=P_x(1,4)$, which follows from Lemma \ref{L2.2}(2). If $3 \in \overline{\varphi}(w)$, then we do a $(1,3)$-swap at $z$ when $P_z(1,3)=P_w(1,3)$, or do a $(1,3)$-swap at both $z$ and $w$ when  $P_z(1,3)\neq P_w(1,3)$. Followed by a $(1,4)$-swap at $z$ which reduces to Subcase 1.2.2.
	Now we only consider $2 \in \overline{\varphi}(w)$. Note that $1 \notin \overline{\varphi}(u)$ since $x$ and $y$ are $(1,3)$-linked. If $2 \in \overline{\varphi}(u)$, then we do a $(1,3)$-swap at $z$, followed by a $(2,3)$-swap at $u$ and another $(1,3)$-swap at $z$. This reduces to Subcase 2.2. If $4 \in \overline{\varphi}(u)$, then we do a $(1,3)$-swap at $z$. Denote the new edge-coloring by $\phiv_1$. Note that $1\in\overline{\varphi_1}(x)\cap\overline{\varphi_1}(z)$. When $P_x(1,4,\phiv_1) \neq P_z(1,4,\phiv_1)$, doing a $(1,4)$-swap at $z$ reduces to Case 1. When $P_x(1,4,\phiv_1)=P_z(1,4,\phiv_1)$, doing a $(1,4)$-swap at $u$. Note that $(y,yv,v,vu,u)$ is a $(1,3)$-chain, contradicting that $P_y(1,3)=P_x(1,3)$, which follows from  Lemma \ref{L2.2}(2). This completes the proof.
\end{proof}

\noindent {\bf Acknowledgements}

This work was supported by Hebei Natural Science Foundation A2026208005 and Foundation of China Scholarship Council 202508130088.

\end{document}